\documentclass[11pt]{amsart}
\usepackage[a4paper, total={6.5in, 8.5in}]{geometry}
\usepackage{enumerate}
\usepackage{amsmath}
\usepackage{amssymb,latexsym}
\usepackage{amsthm}
\usepackage{color}
\usepackage{cancel}
\usepackage{graphicx}
\usepackage{cite}
\usepackage{changes}
\usepackage{hyperref}
\usepackage{comment}
\usepackage{amscd}
\usepackage{mathtools}

\usepackage[foot]{amsaddr}

\newtheorem{theorem}{Theorem}[section]
\newtheorem{problem}[theorem]{Problem}
\newtheorem{lemma}[theorem]{Lemma}
\newtheorem{corollary}[theorem]{Corollary}
\newtheorem{definition}[theorem]{Definition}
\newtheorem{proposition}[theorem]{Proposition}

\newtheorem{remark}[theorem]{Remark}

\newtheorem{open}[theorem]{Problem}

\newcommand{\fqn}{\mathbb{F}_{q^n}}

\newcommand{\cL}{{\mathcal L}}

\newcommand{\F}{{\mathbb F}}

\newcommand{\fq}{{\mathbb F}_{q}}

\newcommand{\la}{\langle}
\newcommand{\ra}{\rangle}

\newcommand{\PG}{\mathrm{PG}}
\newcommand{\N}{\mathrm{N}}

\makeatletter
\renewcommand{\tocsection}[3]{%
  \indentlabel{\@ifnotempty{#2}{\bfseries\ignorespaces#1 #2\quad}}\bfseries#3}
\renewcommand{\tocsubsection}[3]{%
  \indentlabel{\@ifnotempty{#2}{\ignorespaces#1 #2\quad}}#3}

\newcommand\@dotsep{4.5}
\def\@tocline#1#2#3#4#5#6#7{\relax
  \ifnum #1>\c@tocdepth 
  \else
    \par \addpenalty\@secpenalty\addvspace{#2}%
    \begingroup \hyphenpenalty\@M
    \@ifempty{#4}{%
      \@tempdima\csname r@tocindent\number#1\endcsname\relax
    }{%
      \@tempdima#4\relax
    }%
    \parindent\z@ \leftskip#3\relax \advance\leftskip\@tempdima\relax
    \rightskip\@pnumwidth plus1em \parfillskip-\@pnumwidth
    #5\leavevmode\hskip-\@tempdima{#6}\nobreak
    \leaders\hbox{$\m@th\mkern \@dotsep mu\hbox{.}\mkern \@dotsep mu$}\hfill
    \nobreak
    \hbox to\@pnumwidth{\@tocpagenum{\ifnum#1=1\bfseries\fi#7}}\par
    \nobreak
    \endgroup
  \fi}
\AtBeginDocument{%
\expandafter\renewcommand\csname r@tocindent0\endcsname{0pt}
}
\def\l@subsection{\@tocline{2}{0pt}{2.5pc}{5pc}{}}
\makeatother

\title{Multi-orbit cyclic subspace codes and linear sets}
\date{}

\author[Ferdinando Zullo]{Ferdinando Zullo}
\address{Ferdinando Zullo, \textnormal{Dipartimento di Matematica e Fisica, Universit\`a degli Studi della Campania ``Luigi Vanvitelli'', Viale Lincoln, 5, I--\,81100 Caserta, Italy}}
\email{ferdinando.zullo@unicampania.it}

\begin{document}

\maketitle

\begin{center}
    \emph{To the beloved memory of my grandmother Elena.}
\end{center}

\begin{abstract}
Cyclic subspace codes gained a lot of attention especially because they may be used in random network coding for correction of errors and erasures. 
Roth, Raviv and Tamo in 2018 established a connection between cyclic subspace codes (with certain parameters) and Sidon spaces. These latter objects were introduced by Bachoc, Serra and Zémor in 2017 in relation with the linear analogue of Vosper's Theorem.
This connection allowed Roth, Raviv and Tamo to construct large classes of cyclic subspace codes with one or more orbits.
In this paper we will investigate cyclic subspace codes associated to a set of Sidon spaces, that is cyclic subspace codes with more than one orbit. Moreover, we will also use the geometry of linear sets to provide some bounds on the parameters of a cyclic subspace code.
Conversely, cyclic subspace codes are used to construct families of linear sets which extend a class of linear sets recently introduced by Napolitano, Santonastaso, Polverino and the author. This yields large classes of linear sets with a special pattern of intersection with the hyperplanes, defining rank metric and Hamming metric codes with only three distinct weights.
\end{abstract}

\section{Introduction}

Let $k$ be a non-negative integer with $k \leq n$, the set of all $k$-dimensional $\F_q$-subspaces of $\F_{q^n}$, viewed as an $\F_{q}$-vector space, forms a \textbf{Grassmannian space} over $\F_q$, which is denoted by $\mathcal{G}_{q}(n,k)$. A \textbf{constant dimension subspace code} is a subset $C$ of $\mathcal{G}_{q}(n,k)$ endowed with the metric defined as follows \[d(U,V)=\dim_{\F_q}(U)+\dim_{\F_q}(V)-2\dim_{\F_q}(U \cap V),\]
where $U,V \in \mathcal{G}_{q}(n,k)$. This metric is also known as \textbf{subspace metric}.
Subspace codes have been recently used for the error correction in random
network coding, see \cite{KoetterK}. 
The first class of subspace codes studied was the one introduced in \cite{Etzion}, which is known as \textbf{cyclic subspace codes}.
A subspace code $C \subseteq \mathcal{G}_q(n,k)$ is said to be \textbf{cyclic} if for every $\alpha \in \F_{q^n}^*$ and every $V \in C$ then $\alpha V \in C$. 

Let $V \in \mathcal{G}_q(n,k)$, the \textbf{orbit} of $V$ is the set $C_V=\{\alpha V: \alpha \in \F_{q^n}^*\}$, and its size is $(q^n-1)/(q^t-1)$, for some $t$ which divides $n$. More precisely,

\begin{theorem}\cite[Theorem 1]{Otal}\label{th:orbitsize}
Let $U$ be a $k$-dimensional $\fq$-subspace of $\fqn$. Then $\F_{q^d}$ is the largest field such that $U$ is also an $\F_{q^d}$-subspace if and only if the orbit size of $U$ is $\frac{q^n-1}{q^d-1}$.
\end{theorem}

In particular, every orbit of a subspace $V \in \mathcal{G}_q(n,k)$ defines a cyclic subspace code of size $(q^n-1)/(q^t-1)$, for some $t \mid n $. Assume $k > 1$. 
Clearly, a cyclic subspace code generated by an orbit of a subspace $V$ with size $(q^n-1)/(q-1)$ has minimum distance at most $2k-2$ and in \cite{Trautmann} Trautmann, Manganiello, Braun and Rosenthal conjectured the existence of  a cyclic code of size $\frac{q^n-1}{q-1}$ in $\mathcal{G}_q(n,k)$ and minimum distance $2k-2$ for every positive integers $n,k$ such that $k\leq n/2$.

Ben-Sasson, Etzion, Gabizon and Raviv in \cite{BEGR} used subspace polynomials to generate cyclic subspace codes with size $\frac{q^n-1}{q-1}$ and minimum distance $2k-2$, proving that the conjecture is true for any given $k$ and infinitely many values of $n$. Such result was then improved in \cite{Otal} (see also \cite{BEGR,Otal,SantZullo1}). 
Finally, the conjecture was solved in \cite{Roth} for most of the cases, by making use of \emph{Sidon spaces}. They were originally introduced in \cite{BSZ2015}, as an important tool to prove the linear analogue of Vosper's Theorem, which analyzes the equality in the linear analouge of Kneser's theorem proved in \cite{BSZ2018,HouLeungXiang2002}.

An $\fq$-subspace $U$ of $\fqn$ is called a \textbf{Sidon space} if the product of any two elements of $U$ has a unique factorization over $U$, up to multiplying by some elements in $\fq$.
More precisely, $U$ is a Sidon space if for all nonzero $a,b,c,d \in U$, if $ab=cd$, then 
\[ \{a\fq, b\fq\}=\{c\fq, d\fq\}, \]
where if $e\in \fqn$ then $e\fq=\{e \lambda \colon \lambda \in \fq\}$.
Sidon spaces may be seen as the $q$-analogue of \textbf{Sidon set}, that is a subset $A$ in a commutative group with the property that the sums of two elements of $A$ are all distinct except when they coincide because of commutativity. The name comes back to Simon Szidon and was given by Erd\H{o}s in \cite{Erdos}.
Sidon sets have been then intensively studied in several contexts, see e.g.\ \cite{Babai1985,Cill2012,CRC2010Sidon}. 
For a survey the reader is referred to \cite{surverySidon}.

The connection between Sidon spaces and cyclic subspace codes is the following. 

\begin{theorem}\cite[Lemma 34]{Roth}\label{lem:charSidon2}
Let $U$ be an $\fq$-subspace of $\fqn$ of dimension $k$. Then $C_U$ is a cyclic subspace code of size $\frac{q^n-1}{q-1}$ and minimum distance $2k-2$ if and only if $U$ is a Sidon space.
\end{theorem}

In this paper we deal with cyclic subspace codes containing more than one orbit, which  we will refer to as \textbf{multi-orbit cyclic subspace codes}.
More precisely, our aim is to study the following codes. 
Let $U_1,\ldots,U_r$ be $\fq$-subspaces of dimension $k$ in  $\fqn$ and let
\begin{equation}\label{eq:multiorbitcodes} C=\bigcup_{i \in \{1,\ldots,r\}} C_{U_i} \subseteq \mathcal{G}_{q}(n,k)
\end{equation}
with minimum distance $2k-2$ and of size $r\frac{q^n-1}{q-1}$.
These codes have the same minimum distance of the codes associated with Sidon spaces but they are larger, and hence more interesting from a coding theory point of view.
See \cite{code2,code1,code3} for recent constructions.

The code $C$ is uniquely defined by the subspaces $U_1,\ldots,U_r$ and therefore we give the following two definitions.
The set $\{U_1,\ldots,U_r\}$ of a cyclic subspace code as in  \eqref{eq:multiorbitcodes} will said to be a set of \textbf{representatives} of $C$, whereas if also its minimum distance is $2k-2$ we will call it a \textbf{multi-Sidon space}, in connection with the correspondence between cyclic subspace codes with these parameters with one orbit and Sidon spaces. Equivalently, the set $\{U_1,\ldots,U_r\}$ is a multi-Sidon space if
\[\dim_{\fq} (U_i \cap \alpha U_j) \leq 1\]
for every $\alpha \in \fqn$ and $i,j \in \{1,\ldots,r\}$ with $i\ne j$ and for every $\alpha \in \fqn\setminus\fq$ if $i=j$.

Some of the results we will show do not require that all the subspaces have the same dimension $k$. In this case, the associated code will be still cyclic but not a constant dimension code.

In this paper, we first investigate multi-orbit cyclic subspace codes by analyzing their properties and providing characterizations, involving the Cartesian product of a set of representatives. Then we derive a canonical form for multi-orbit cyclic codes in $\mathcal{G}_q(n,n/2)$ (and hence of multi-Sidon spaces), when $n$ is even, through some linearized polynomials jointly with direct conditions that can be checked to establish whether or not a set of subspaces is a multi-Sidon space. Then we used these conditions to find a family of multi-orbit cyclic subspace codes defined by linearized monomials, extending the family presented by Roth, Raviv and Tamo in \cite{Roth}, and we can also solve the equivalence issue among the codes belonging to this larger family.
We gave a geometric interpretation of the Sidon property in terms of linear sets. This connection yields some new bounds on the parameters involved in a multi-Sidon space and hence on the associated subspace code. 
With the aid of multi-orbit cyclic subspace codes, we study a special class of linear sets in $\PG(r-1,q^n)$ which may be defined by $r$ points in independent position, which naturally extends those investigated in \cite{NPSZ2021} with Napolitano, Santonastaso and Polverino. In particular, we provide bounds on the rank of such linear sets and we then concentrate our attention on those that have maximum rank. Constructions of such linear sets can be obtained by using multi-orbit cyclic subspace codes and we then analyze the equivalence issue among those. 
Moreover, when such linear sets have rank $n$ they can be described by projection maps, as shown in the appendix. 
Some of these properties are natural generalization of the results in \cite{NPSZ2021}, but unlike what happens in the projective line, the duals of the above linear sets of maximum rank have a special pattern of intersection with the hyperplanes. This means that such linear sets can be used to define linear rank metric codes and linear Hamming metric codes with only three weights (yielding also to \emph{almost MRD codes}) and for which one can completely determine their weight distribution.
We conclude the paper by showing possible problems that could be of interest for the reader.

\section{Preliminaries}

\subsection{Linearized polynomials}

Let $s$ be a positive integer such that $\gcd(s,n)=1$.
A $q^s$-\emph{linearized polynomial} (for short $q^s$-\emph{polynomial}) over $\F_{q^n}$ is a polynomial of the form
\[f(x)=\sum_{i=0}^{\ell} a_i x^{q^{is}},\]
where $a_i\in \F_{q^n}$ and $\ell$ is a positive integer.
Furthermore, if $a_\ell \neq 0$ we say that $\ell$ is the $q^s$-\emph{degree} of $f(x)$.
We will denote by $\mathcal{L}_{n,s}$ the set of all $q^s$-polynomials over $\F_{q^n}$ with $q^s$-degree less than $n$ (or $\mathcal{L}_{n}$ if $s=1$). Together with the classical sum of polynomials, the composition modulo $x^{q^n}-x$ and the scalar multiplication by an element in $\fq$, $\mathcal{L}_{n,s}$ forms an $\fq$-algebra isomorphic to the algebra of $\fq$-linear endomorphisms of $\fqn$.
For this reason, we shall identify the elements of $\cL_{n,s}$ with the
endomorphisms of $\F_{q^n}$ they represent and hence we will also speak of \emph{kernel} and \emph{rank} of a $q^s$-polynomial.
Clearly, the kernel of $f(x)\in \cL_{n,s}$ coincides with the set of the roots of $f(x)$ over $\F_{q^n}$.

We now recall an important result on linearized polynomials which will play a crucial role in the construction of examples of cyclic subspace codes.
This puts together \cite[Lemma 3.2]{Guralnick} and \cite[Theorem 10]{GQ2009x} (see also \cite[Theorem 5]{GQ2009}).

\begin{theorem}\label{Gow}
Consider
\[f(x)=a_0x+a_1x^{q^s}+\cdots+a_{k-1}x^{q^{s(k-1)}}+a_kx^{q^{sk}}\in \mathcal{L}_{n,s},\]
with $k\leq n-1$ and let $a_0,a_1,\ldots,a_k$ be elements of $\F_{q^n}$ not all of them zero. Then 
\[ \dim_{\F_q}(\ker (f(x)))\leq k. \]
Moreover, if $\dim_{\F_q}(\ker (f(x)))=k$ then $\N_{q^n/q}(a_0)=(-1)^{nk}\N_{q^n/q}(a_k)$.
\end{theorem}

An important example of a linearized polynomial is given by the trace function, which can be defined as
\[ \mathrm{Tr}_{q^n/q}(x)=x+x^q+\ldots+x^{q^{n-1}}. \]

For more details on linearized polynomials we refer to \cite[Chapter 3, Section 4]{lidl_finite_1997}.

\subsection{Linear sets}

A point set $L$ of $\Lambda=\PG(V,\F_{q^n})\allowbreak=\PG(r-1,q^n)$ is said to be an \textbf{$\F_q$-linear set} of $\Lambda$ of rank $k$ if it is defined by the non-zero vectors of a $k$-dimensional $\F_q$-vector subspace $U$ of $V$, i.e.
\[L=L_U:=\{\la {\bf u} \ra_{\mathbb{F}_{q^n}} : {\bf u}\in U\setminus \{{\bf 0} \}\}.\]
We denote the rank of an $\fq$-linear set $L_U$ by $\mathrm{Rank}(L_U)$.
For any subspace $S=\PG(Z,\mathbb{F}_{q^n})$ of $\Lambda$, the \textbf{weight} of $S$ in $L_U$ is defined as $w_{L_U}(S)=\dim_{\mathbb{F}_q}(U\cap Z)$.
If $N_i$ denotes the number of points of $\Lambda$ having weight $i\in \{0,\ldots,k\}$  in $L_U$, the following relations hold:
\begin{equation}\label{eq:card}
    |L_U| \leq \frac{q^k-1}{q-1},
\end{equation}
\begin{equation}\label{eq:pesicard}
    |L_U| =N_1+\ldots+N_k,
\end{equation}
\begin{equation}\label{eq:pesivett}
    N_1+N_2(q+1)+\ldots+N_k(q^{k-1}+\ldots+q+1)=q^{k-1}+\ldots+q+1.
\end{equation}

We also recall that two linear sets $L_U$ and $L_W$ of $\PG(r-1,q^n)$ are said to be \textbf{$\mathrm{P\Gamma L}$-equivalent} (or simply \textbf{equivalent}) if there is an element $\varphi$ in $\mathrm{P\Gamma L}(r,q^n)$ such that $L_U^{\varphi} = L_W$.
In general, it is very hard to determine  whether or not two linear sets are projective equivalent, see e.g. \cite{BoPol,CsMP,CsZan}.

Another notion that we will need is that of the dual of a linear set.
Let $L_U$ be an $\fq$-linear set of rank $k$ in $\PG(V,\fqn)=\PG(r-1,q^n)$.
Let $\sigma \colon V \times V \rightarrow \fqn$ be a nondegenerate reflexive sesquilinear form on the $\fqn$-vector space $V$ and consider $\sigma'=\mathrm{Tr}_{q^n/q} \circ \sigma$, which is a nondegenerate reflexive sesquilinear form on $V$ seen as an $\fq$-vector space of dimension $rn$. Then we may consider $\perp$ and $\perp'$ as the orthogonal complement maps defined by $\sigma$ and $\sigma'$, respectively, and $\tau$ and $\tau'$ as the polarities of $\PG(V,\fqn)$ and $\PG(V,\fq)$ defined by $\perp$ and $\perp'$, respectively. Then $L_U^\tau=L_{U^{\perp'}}$ is an $\fq$-linear set in $\PG(V,\fqn)$ of rank $rn-k$ that is called the \textbf{dual linear set of $L_U$} with respect to the polarity $\tau$.
As proved in \cite[Proposition 2.5]{Polverino}, this linear set does not depend on the polarity $\tau$. Indeed, let  $\sigma_1 \colon V \times V \rightarrow \fqn$ be a nondegenerate reflexive sesquilinear form on the $\fqn$-vector space $V$ and let $\tau_1$ the related polarity, then $L_U^\tau$ and $L_U^{\tau_1}$ are  $\mathrm{P\Gamma L}(2,q^n)$-equivalent;
see also \cite{trans}.

Moreover, the following result relates the weights of the subspaces with respect to the linear set and its dual.

\begin{proposition}\cite[Property 2.6]{Polverino}\label{prop:weightdual}
Let $L_U$ be an $\fq$-linear set of rank $k$ in $\PG(V,\fqn)=\PG(r-1,q^n)$. If $S=\PG(Z,\fqn)$ is an $s$-dimensional projective subspace then
\[ w_{L_U^\tau}(S^\tau)-w_{L_U}(S)=rn-k-(s+1)n. \]
\end{proposition}

For further details on linear sets see \cite{LavVdV,Polverino}.

\section{Structure and properties of multi-orbit cyclic subspace codes}

In this section we will analyze properties of multi-orbit cyclic subspace codes, paying special attention to their algebraic description.
Interesting results have been proved in this direction for the case of Sidon spaces.

Let $U$ and $V$ be two $\fq$-subspaces of $\fqn$. Denote by $\langle U^2\rangle$ the $\fq$-span of $\{st\colon s,t \in U \}$, $U^{-1}=\{u^{-1} \colon u \in U\setminus\{0\}\}$ and $U\cdot V=\{uv \colon u \in U, v \in V\}$.

Bachoc, Serra and Z\'emor proved a lower bound on the dimension of a Sidon space in \cite[Theorem 18]{BSZ2015} and hence the following holds.

\begin{theorem}
If $U$ is a Sidon space in $\fqn$ of dimension $k\geq 3$, then
\[ 2k\leq \dim_{\fq}(\langle U^2\rangle )\leq {k+1 \choose 2}. \]
\end{theorem}

Clearly, this result implies that if a cyclic subspace code $C_U$ has minimum distance $2k-2$, where $k=\dim_{\fq}(U)$, then $2k\leq n$.

We can hence apply the above result to all the subspaces of a multi-Sidon space, obtaining the following bounds.

\begin{corollary}
Let $\{U_1,\ldots,U_r\}$ is a multi-Sidon space of $\fqn$.
Let $k_i=\dim_{\fq}(U_i)\geq 3$ for any $i \in \{1,\ldots,r\}$. Then
\[ 2\sum_{i=1}^r k_i \leq \dim_{\fq}(\langle U_1^2\rangle \times \ldots \times \langle U_r^2\rangle )\leq \sum_{i=1}^r {k_i+1 \choose 2}. \]
In particular, $k_i \leq n/2$ for each $i \in \{1,\ldots,r\}$.
\end{corollary}


A multi-Sidon space $\{U_1,\ldots,U_r\}$ of $\fqn$, with $k_i=\dim_{\fq}(U_i)$ for any $i \in \{1,\ldots,r\}$, is said to be \textbf{maximum} if $n$ is even and $k_i=n/2$ for every $i \in \{1,\ldots,r\}$. The interest for maximum multi-Sidon spaces arise from the fact that the associated codes are those with the largest minimum distance, that is $n-2$.

In \cite{Roth}, it has been described the property of being a Sidon space with an algebraic flavour, using the uniqueness of the factorization of the product of two elements which are in two distinct subspaces.

\begin{lemma}\cite[Lemma 36]{Roth}\label{lem:charalg}
Let $U, V$ be two distinct $\fq$-subspaces of dimension $k$ of $\fqn$. Then the following conditions are equivalent 
\begin{enumerate}
    \item $\dim_{\fq} (U \cap \alpha V) \leq 1$, for any $\alpha \in \F_{q^n}$;
    \item For any nonzero $a,c \in U$ and nonzero $b,d \in V$, the equality $ab=cd$ implies that $a\F_q=c \F_q$ and $b \F_q=d \F_q$.
\end{enumerate}
\end{lemma}

Lemma \ref{lem:charalg} can be also applied to a set of subspaces $\mathcal{U}$, extending the uniqueness of the factorization of the product of two elements which are in two distinct subspaces in $\mathcal{U}$.

\begin{corollary}\label{cor:UiUj}
Let $\{U_1,\ldots,U_r\}$ be a set of $\fq$-subspaces of $\fqn$. 
Then the following conditions are equivalent 
\begin{enumerate}
    \item $\dim_{\fq} (U_i \cap \alpha U_j) \leq 1$, for any $\alpha \in \F_{q^n}$ and $i,j \in \{1,\ldots,r\}$ with $i\ne j$;
    \item For any nonzero $a,c \in U_i$ and nonzero $b,d \in U_j$, the equality $ab=cd$ implies that $a\F_q=c \F_q$ and $b \F_q=d \F_q$, for any $i,j \in \{1,\ldots,r\}$ with $i\ne j$.
\end{enumerate}
\end{corollary}

The property of being a multi-Sidon space is inherited by the Cartesian product of such subspaces as follows.
To this aim, let $\mathbf{e}_i\in \fqn^r$ be the vector whose $i$-th component is $1$ and all the others are zero.

\begin{theorem}\label{th:character}
Let $\{U_1,\ldots,U_r\}$ be a set of $\fq$-subspaces of $\fqn$. 
Then the following are equivalent:
\begin{itemize}
\item[i)] the $\fq$-subspace $U=U_1\times \ldots \times U_r$ of $\fqn^r$ is such that $\dim_{\fq}(U\cap \langle \mathbf{v}\rangle_{\fqn})\leq 1$ for any $\mathbf{v}\in \fqn^r$ such that $\langle\mathbf{v}\rangle_{\fqn}\notin \{\langle \mathbf{e}_i \rangle_{\fqn} \colon i \in \{1,\ldots,r\}\}$;
\item[ii)] $U_i\cdot U_i^{-1}\cap U_j\cdot U_j^{-1}=\fq$ for every $i\ne j$;
\item[iii)] $\dim_{\fq}(U_i\cap \alpha U_j)\leq 1$, for every $\alpha \in \fqn$ and $i,j \in \{1,\ldots,r\}$ such that $i\ne j$.
\end{itemize}
\end{theorem}
\begin{proof}
Consider a nonzero $\lambda \in U_i\cdot U_i^{-1}\cap U_j\cdot U_j^{-1}$ for some $i,j \in \{1,\ldots,r\}$ with $i\ne j$ and assume that i) holds.
Then there exist $x,x' \in U_i$ and $y,y' \in U_j$ such that 
\begin{equation}\label{eq:lambda} \lambda=\frac{x}{x'}=\frac{y}{y'}.
\end{equation}
If $\lambda$ were not in $\fq$, then we would have a contradiction. 
Indeed, consider the vector $\mathbf{u}$ whose $i$-th entry is $x$, $j$-th entry is $y$ and all the other entries are zero and the vector $\mathbf{u}'$ whose $i$-th entry is $x'$, $j$-th entry is $y'$ and all the other entries are zero, whereas $x,y,x',y'$ are nonzero.
By \eqref{eq:lambda} we obtain $\mathbf{u}=\lambda \mathbf{u}'$, a contradiction to a).
Now, suppose that ii) holds.
Let $\mathbf{v}\in U$ such that $\langle\mathbf{v}\rangle_{\fqn}\notin \{\langle \mathbf{e}_i \rangle_{\fqn} \colon i \in \{1,\ldots,r\}\}$ and $\dim_{\fq}(U\cap \langle \mathbf{v}\rangle_{\fqn})\geq 2$. Then there exist $i,j \in \{1,\ldots,r\}$ such that $v_i\ne 0$ and $v_j\ne 0$ and $\rho \in \fqn\setminus \fq$ such that $\rho \mathbf{v}=\mathbf{u}$ with $\mathbf{u}\in U$. In particular $u_i,u_j$ are non zero and
\[ \rho=\frac{u_i}{v_i}=\frac{u_j}{v_j}\in U_i\cdot U_i^{-1}\cap U_j\cdot U_j^{-1}, \]
and hence we get a contradiction since $\rho \notin \fq$.
Therefore, we have proved that i) and ii) are equivalent.
We now conclude the proof by proving the equivalence between ii) and iii).
Suppose that $U_i\cdot U_i^{-1} \cap U_j \cdot U_j^{-1}=\F_q$ for every $i,j \in \{1,\ldots,r\}$ with $i\ne j$.
By contradiction, assume the existence of $\alpha \in \F_{q^n}^*$ such that $\dim_{\fq}(U_i \cap \alpha U_j) \geq 2$ for some $i,j \in \{1,\ldots,r\}$ with $i\ne j$.
Let $a_1$ and $a_2$ be two $\fq$-linearly independent elements in $U_i \cap \alpha U_j$.
Moreover, since $a_1,a_2 \in \alpha U_j$, there exist $b_1,b_2 \in U_j\setminus\{0\}$ such that $a_1=\alpha b_1$ and $a_2=\alpha b_2$, so that 
\[a_1a_2^{-1}=b_1b_2^{-1}\in U_i\cdot U_i^{-1}\cap U_j\cdot U_j^{-1}\] 
and hence, by assumption, $a_1a_2^{-1} \in \fq$, which is a contradiction.
Now, assume that $\dim_{\fq}(U_i\cap \alpha U_j)\leq 1$, for every $\alpha \in \fqn$ and $i,j \in \{1,\ldots,r\}$ such that $i\ne j$.
Let $\lambda \in U_i\cdot U_i^{-1} \cap U_j \cdot U_j^{-1} \setminus\{0\}$. Then there exist $a,c \in U_i\setminus\{0\}$ and $b,d \in U_j\setminus\{0\}$ such that $\lambda=ac^{-1}=db^{-1}$.
Then $a,c \in U_i \cap \beta U_j$ where $\beta=a/d=c/b$. Since $\dim_{\fq}(U_i \cap \beta U_j)\leq 1$, it follows that $\lambda=ac^{-1} \in \F_q$.
\end{proof}

The above theorem give equivalent conditions to the property of being a multi-Sidon space when we consider a set of Sidon spaces. This result will also be very important in studying the equivalence among linear sets defined by $r$ independent points as we will see later.

We can now describe a canonical form of cyclic subspace codes with minimum distance $n-2$.

\begin{theorem}\label{th:canonicalform}
Let $n=2k\geq 4$, suppose that $\mathcal{U}=\{U_1,\ldots,U_r\}$ is a set of $\fq$-subspaces in $\fqn$ with dimension $k$ and let $C=\bigcup_{i=1}^r C_{U_i}$. Then there exists a set of representatives for $C$
\[ \{W_{f_1,\eta_1},\ldots,W_{f_r,\eta_r}\}, \]
where $f_1(x),\ldots,f_r(x) \in \mathcal{L}_{k}$, $\eta_1,\ldots,\eta_r \in \fqn\setminus\F_{q^k}$ and
\[ W_{f_i,\eta_i}=\{x+\eta_i f_i(x) \colon x \in \F_{q^k}\}, \]
for every $i \in \{1,\ldots,r\}$.
Let $\eta_i=A_{i,j}\eta_j+B_{i,j}$ and $\eta_i^2=a_i\eta_i+b_i$ with $A_{i,j},B_{i,j},a_i,b_i \in \F_{q^k}$ for any $i,j \in \{1,\ldots,r\}$.
Moreover, the minimum distance of $C$ is $n-2$ if and only if for every $i,j \in \{1,\ldots,r\}$ and $\alpha_0,\alpha_1 \in \F_{q^k}$ with $(\alpha_0,\alpha_1)\ne (0,0)$ the following linearized polynomial in $\mathcal{L}_{k}$ 
\[ F_{i,j}(x)=f_i(\alpha_0 x) + f_i(\alpha_1 A_{j,i} b_i f_j(x))+f_i(\alpha_0 B_{j,i} f_j(x))-\alpha_1 x -\alpha_0 A_{j,i} f_j(x)-\alpha_1 A_{j,i} a_i f_j(x)-\alpha_1 B_{j,i} f_j(x) \]
if $i\ne j$, and 
\[ F_{i,i}(x)=f_i(\alpha_0 x)+f(\alpha_1 b_i f_i(x))-\alpha_1 x -\alpha_0 f_j(x)-\alpha_1 a_i f_i(x) \]
if $i=j$, is either the zero polynomial or it has at most $q$ roots over $\F_{q^k}$.
\end{theorem}
\begin{proof}
Let $\eta_1,\ldots,\eta_r \in \fqn\setminus\F_{q^k}$.
Clearly, $\{1,\eta_i\}$ is an $\F_{q^k}$-basis of $\fqn$ for every $i$.
For each $i \in \{1,\ldots,r\}$, there exists $\lambda_i \in \fqn^*$ such that
\[ \lambda_i U_i \cap \eta_i \F_{q^k}=\{0\}. \]
By contradiction assume that $\lambda U_i \cap \eta_i \F_{q^k}\ne\{0\}$ for each $\lambda \in \fqn^*$. 
Consider the Desarguesian spread $\mathcal{D}=\{ \langle \mathbf{v} \rangle_{\F_{q^n}} \colon \mathbf{v}\in \fqn\times\fqn\setminus\{\mathbf{0}\} \}=\{\langle (1,\xi) \rangle_{\fqn} \colon \xi \in \fqn\}\cup \{ \langle (0,1) \rangle_{\fqn} \}$ of $\fqn\times \fqn$.
The $\fq$-subspace $U_i\times \eta_i \F_{q^k}=\{(a,b)\colon a \in U_i,b \in \eta_i \F_{q^k}\}$ meets all the elements of $\mathcal{D}$ in at least one non-zero vector.
Indeed, by the hypothesis 
\[ \{b/a \colon a \in U_i\setminus\{0\},b \in \eta_i \F_{q^k}\}=\fqn, \]
so that for every $\xi \in \fqn$ there exist $a \in U_i\setminus\{0\}$ and $b \in \eta_i \F_{q^k}$ such that $\xi=b/a$ and so
\[ \langle (1,\xi) \rangle_{\fqn} \cap (U_i\times \eta_i \F_{q^k}) \supseteq \langle (a,b)\rangle_{\fq}.  \]
Clearly, $\langle (0,1) \rangle_{\fqn} \cap (U_i\times \eta_i \F_{q^k})=U_i$ and hence $U_i\times \eta_i \F_{q^k}$ meets all the elements of $\mathcal{D}$ in at least one non-zero vector.
Since $\dim_{\fq}(U_i\times \eta_i \F_{q^k})=n$ this is a contradiction, see e.g.\ \cite[(3) Proposition 2.2]{Polverino}.
Therefore, there exists $\lambda_i \in \fqn^*$ such that
\[ \lambda_i U_i \cap \eta_i \F_{q^k}=\{0\}. \]
This implies that for every $i \in \{1,\ldots,r\}$, since $\{1,\eta_i\}$ is an $\F_{q^k}$-basis of $\fqn$, there exist $r$ linearized polynomials $f_i(x)\in \mathcal{L}_{k}$ such that
\[ \lambda_i U_i = W_{f_i,\eta_i}. \]
Now, let $i,j \in \{1,\ldots,r\}$ with $i\ne j$ and consider $w \in W_{f_i,\eta_i}\cap \alpha W_{f_j,\eta_j}$, for some $\alpha \in \fqn^*$.
Then there exist $u,v \in \F_{q^k}$ such that
\begin{equation}\label{eq:conditions} w=u+\eta_i f_i(u)=\alpha(v+\eta_j f_j(v)). \end{equation}
Since $\{1,\eta_i\}$ is an $\F_{q^k}$-basis of $\fqn$, $\eta_i^2=a_i\eta_i+b_i$ and $\eta_j=A_{j,i}\eta_i+B_{j,i}$, we get that $\alpha=\alpha_0+\alpha_1\eta_i$ for some $\alpha_0,\alpha_1 \in \F_{q^k}$ and \eqref{eq:conditions} may be rewritten as
\[ 
\left\{
\begin{array}{ll}
u=\alpha_0 v +\alpha_1 A_{j,i} b_i f_j(v)+\alpha_0 B_{j,i} f_j(v),\\
f_i(u)=\alpha_1 v +\alpha_0 A_{j,i} f_j(v)+\alpha_1 A_{j,i} a_i f_j(v)+\alpha_1 B_{j,i} f_j(v),
\end{array}
\right.
\]
from which we get
\[ f_i(\alpha_0 v) + f_i(\alpha_1 A_{j,i} b_i f_j(v))+f_i(\alpha_0 B_{j,i} f_j(v))=\alpha_1 v +\alpha_0 A_{j,i} f_j(v)+\alpha_1 A_{j,i} a_i f_j(v)+\alpha_1 B_{j,i} f_j(v). \]
The assertion then follows from the fact that the minimum distance of $C$ is $n-2$ and noting that the above polynomial is the zero polynomial if and only if $W_{f_i,\eta_i}=\alpha W_{f_j,\eta_j}$, a contradiction to the fact that $i\ne j$.
If $i=j$, the result can be obtained by repeating the previous argument (see also \cite[Corollary 4.7]{NPSZ2021}) noting that the the condition \eqref{eq:conditions} gives rise to the following system
\[ 
\left\{
\begin{array}{ll}
u=\alpha_0 v +\alpha_1 b_i f_i(v),\\
f_i(u)=\alpha_1 v +\alpha_0  f_i(v)+\alpha_1 a_i f_i(v).
\end{array}
\right.
\]
\end{proof}

\begin{remark}
Clearly, since 
\[\dim_{\fq}(W_{f_i,\eta_i}\cap \alpha W_{f_j,\eta_j})=\dim_{\fq}(\alpha^{-1}W_{f_i,\eta_i}\cap W_{f_j,\eta_j})\] 
for $\alpha \in \fqn^*$, we could just assume that in Theorem \ref{th:canonicalform} the condition on the polynomials $F_{i,j}(x)$ hold for every $i,j$ with $i\leq j$. 
\end{remark}

In particular, the canonical form for a Sidon space in $\fqn$ of dimension $n/2$ is the following.

\begin{corollary}
Let $n=2k\geq 4$ and let $U$ be an $\fq$-subspace of $\fqn$ with dimension $k$. Then $U$ is equivalent to $W_{f,\eta}=\{x+\eta f(x) \colon x \in \F_{q^k}\}$,
where $f(x) \in \mathcal{L}_{k}$ and $\eta \in \fqn\setminus\F_{q^k}$.
Let $\eta^2=a\eta+b$ with $a,b \in \F_{q^k}$.
Then, $U$ is a Sidon space if and only if for every $\alpha_0,\alpha_1 \in \F_{q^k}$ with $(\alpha_0,\alpha_1)\ne(0,0)$ the following linearized polynomial in $\mathcal{L}_{k}$
\[ F(x)=f(\alpha_0 x) + f(\alpha_1 b f(x))-\alpha_1 x -\alpha_0  f(x)-\alpha_1  a f(x) \]
is either the zero polynomial or it has at most $q$ roots over $\F_{q^k}$.
\end{corollary}

Now we show some examples of multi-orbit cyclic subspace codes with minimum distance $n-2$, making use of Theorem \ref{th:canonicalform}.

\begin{theorem}\label{prop:examplepseudo}
Let $s$ be a positive integer coprime with $k\geq 2$ and $n=2k$, let $\xi \in \fqn\setminus \F_{q^k}$ and let $f_i(x)=\mu_i x^{q^s} \in \mathcal{L}_{k}$ for $i \in \{1,\ldots,r\}$ such that $r\leq q-1$,  $\N_{q^k/q}(\mu_i)\ne\N_{q^k/q}(\mu_j)$ and $\N_{q^k/q}(\mu_i \mu_j \xi^{q^k+1})\ne 1$ for every $i\ne j$.
Then 
\[C=\bigcup_{i=1}^rC_{W_{f_i,\xi}}\] 
is a cyclic subspace code of size $r\frac{q^n-1}{q-1}$ and minimum distance $n-2$.
\end{theorem}
\begin{proof}
Let $x^2-ax-b$ be the minimal polynomial of $\xi$ over $\F_{q^k}$. 
Clearly, $b=-\xi^{q^k+1}$.
Let $i,j \in \{1,\ldots,r\}$, then the polynomials $F_{i,j}(x)$ of Theorem \ref{th:canonicalform} read as follows
\[ F_{i,j}(x)=\mu_i\alpha_0^{q^s}x^{q^s}+\mu_i \alpha_1^{q^s}b^{q^s} \mu_j^{q^s} x^{q^{2s}}-\alpha_1 x-\alpha_0 \mu_j x^{q^s}-\alpha_1  a\mu_j x^{q^s},\]
since the $\eta_i$'s in Theorem \ref{th:canonicalform} are chosen to be equal to $\xi$ and hence $A_{j,i}=1$ and $B_{j,i}=0$, for every $i$ and $j$.
If $F_{i,j}(x)$ is not the zero polynomial, then it can be seen as $q^s$-polynomial with $q^s$-degree at most two. 
If the coefficient of $x^{q^{2s}}$ is zero, then by Theorem \ref{Gow} $\dim_{\fq}(\ker(F_{i,j}(x)))\leq 1$.
If the coefficient of $x^{q^{2s}}$ is non-zero and $F_{i,j}(x)$ admits $q^2$ roots, Theorem \ref{Gow} implies that
\[ \N_{q^k/q}(\mu_i \alpha_1^{q^s} \mu_j^{q^s} b^{q^s})=\N_{q^k/q}(-\alpha_1), \]
a contradiction to the fact that $\N_{q^k/q}(\mu_i \mu_j \xi^{q^k+1})\ne 1$ for every $i$ and $j$.
So, by Theorem \ref{th:canonicalform} the assertion follows.
\end{proof}

\begin{remark}
Note that in the above theorem $r$ cannot reach $q-1$ when $q\geq 4$.
Indeed, suppose that $r=q-1$, then
\[ \{ \N_{q^k/q}(\mu_i) \colon i \in \{1,\ldots,r\} \}=\mathbb{F}_q^*, \]
and so when $q\geq 4$
\[ \{ \N_{q^k/q}(\mu_i\mu_j) \colon i,j \in \{1,\ldots,r\}, i \ne j \}=\mathbb{F}_q^*. \]
Therefore,
\[ \{ \N_{q^k/q}(\mu_i\mu_j) \N_{q^k/q}(\xi^{q^k+1}) \colon i,j \in \{1,\ldots,r\}, i \ne j \}=\mathbb{F}_q^*, \]
and hence there exist $i,j \in \{1,\ldots,r\}$ with $i\ne j$ such that $ \N_{q^k/q}(\mu_i\mu_j) \N_{q^k/q}(\xi^{q^k+1})=1$.
If $q=3$, $r=2$ can be reached. Indeed, it is enough to take $\mu_1,\mu_2 \in \mathbb{F}_{q^k}$ and $\xi \in \mathbb{F}_{q^{n}}\setminus\mathbb{F}_{q^k}$ such that $\N_{q^k/q}(\mu_1)=1,  \N_{q^k/q}(\mu_2)=-1$ and $\N_{q^k/q}(\xi^{q^k+1})=1$.
\end{remark}

The first construction of cyclic codes with these parameters was given in \cite{Roth}.
In the above result we have extended it to the case $s>1$.
As we will see later, this yields to a much larger class of new codes.
Choosing the $\gamma_i$'s as in the paper \cite{Roth}, we will extend \cite[Lemma 38]{Roth} as follows.


\begin{corollary} \label{teo:subspacecodesidon}
For a prime power $q \geq 3$ and a positive integer $k\geq 2$, let $w$ be a primitive element of $\F_{q^k}$ and let $s$ be a positive integer such that $\gcd(s,k)=1$. Let $b \in \F_{q^k}$ be such that the polynomial $p(x)=x^2+bx+w$ is irreducible over $\F_{q^k}$ (such $b$ always exist). For $n=2k$, let $\gamma_0\in \F_{q^n}$ be a root of $p(x)$. For $i \in \{0,1,\ldots,\tau-1\}$, where $\tau=\lfloor(q-1)/2 \rfloor$, let $\gamma_i=w^i\gamma_0$ and let 
$$V_i=\{u+u^{q^s} \gamma_i:u \in \F_{q^k}\}.$$
The set 
\[G_{n,s}=\bigcup_{i \in \{0,1,\ldots,\tau-1\}} C_{V_i}\subseteq \mathcal{G}_q(n,k)\] 
is a subspace code of size $\tau \cdot (q^n-1)/(q-1)$ and minimum distance $n-2$.
\end{corollary}
\begin{proof}
Let $\mu_i=w^i$ for any $i$ and $\gamma_0=\xi$.
Note that $\gamma_0^{q^k+1}=\xi^{q^k+1}=w$.
Let $i,j \in \{0,\ldots,\tau-1\}$. Note that $w^i$ and $w^j$ are such that $\N_{q^k/q}(w^i)\ne \N_{q^k/q}(w^j)$, since $w$ is a primitive element and $|i-j|<q-1$.
Moreover, 
\[ \N_{q^k/q}(\mu_i\mu_j\xi^{q^k+1})=\N_{q^k/q}(w)^{i+j+1}, \]
which cannot be one, again since $w$ is primitive and $i+j+1<q-1$.
So, we can apply Theorem \ref{prop:examplepseudo} to get the assertion.
\end{proof}

Adding to the list of $V_i$'s also the subfield $\F_{q^k}$ extends the code $G_{n,s}$ preserving the minimum distance but keeping a large number of codewords.

\begin{proposition}\label{prop:subfield}
For a prime power $q \geq 3$ and a positive integer $k\geq 2$, let $w$ be a primitive element of $\F_{q^k}$ and let $s$ be a positive integer such that $\gcd(s,k)=1$. Let $b \in \F_{q^k}$ be such that the polynomial $p(x)=x^2+bx+w$ is irreducible over $\F_{q^k}$ (such $b$ always exist). For $n=2k$, let $\gamma_0\in \F_{q^n}$ be a root of $p(x)$. For $i \in \{0,1,\ldots,\tau-1\}$, where $\tau=\lfloor(q-1)/2 \rfloor$, let $\gamma_i=w^i\gamma_0$ and let 
$$V_i=\{u+u^{q^s} \gamma_i:u \in \F_{q^k}\}.$$
Then 
\[ \dim_{\fq}(\F_{q^k}\cap \alpha V_i)\leq 1 \]
for every $i\in\{0,\ldots,\tau-1\}$.
In particular, the set
\[ \overline{G}_{n,s}= \bigcup_{i \in \{0,1,\ldots,\tau-1\}} C_{V_i} \cup C_{\F_{q^k}} \subseteq \mathcal{G}_q(n,k)\] 
is a subspace code of size $\tau \cdot (q^n-1)/(q-1)+q^k+1$ with minimum distance $n-2$.
\end{proposition}
\begin{proof}
Clearly, $C$ is a cyclic code and its size is $\tau \cdot (q^n-1)/(q-1)+q^k+1$ because of Theorem \ref{th:orbitsize} and Corollary \ref{teo:subspacecodesidon}.
Let $\alpha \in \fqn^*$.
We now first show that
\[ \dim_{\fq}(\F_{q^k}\cap \alpha V_i)\leq 1. \]
Then there exist $\alpha_0,\alpha_1 \in \F_{q^k}$ such that $\alpha=\alpha_0+\alpha_1 \gamma_0$. 
Let $v \in \F_{q^k}$, we look for solutions in $u \in \F_{q^k}$ such that
\[ v=\alpha(u+u^{q^s}\gamma_i). \]
Using that $\gamma_0^2=a\gamma_0+b$, for some $a,b \in \F_{q^k}$, and that $\alpha=\alpha_0+\alpha_1 \gamma_0$, the above relation reads as follows
\[ 
\left\{
\begin{array}{ll}
\alpha_0 u +b \alpha_1 u^{q^s} w^i=v,\\
\alpha_1 u+\alpha_0 w^i u^{q^s} +w^i\alpha_1 a u^{q^s}=0.
\end{array}
\right.
\]
Since $\gcd(s,k)=1$, $\alpha_1 u+\alpha_0 w^i u^{q^s} +w^i\alpha_1 a u^{q^s}$ can be seen as a non-zero $q^s$-polynomial in $u$ over $\F_{q^k}$ of $q^s$-degree at most one. Therefore, by Theorem \ref{Gow} it follows that the number of $u \in \F_{q^k}$ which are solutions of the above system is at most $q$ and hence 
\[ \dim_{\fq}(\F_{q^k}\cap \alpha V_i)\leq 1. \]

Now, note that for any $\alpha \notin \F_{q^k}$ we have
\[ \dim_{\fq}(\F_{q^k}\cap \alpha \F_{q^k})=0. \]
By Corollary \ref{teo:subspacecodesidon}, we have that
\[ \dim_{\fq}(V_i\cap \alpha V_j)\leq 1, \]
for each $i,j \in \{0,\ldots,\tau-1\}$.
So, the minimum distance of $C$ is $n-2$.
\end{proof}

\section{Equivalences of multi-orbit cyclic subspace codes}

In \cite{Trautmann2} Horlemann-Trautmann initiated the study of the equivalence for subspace codes in a very general setting. More recently in \cite{Heideequiv} Gluesing-Luerssen and Lehmann investigate the case of cyclic orbit codes, that is $G$-orbits of a subspace $U$ of $\F_{q^n}$ with $G$ a Singer cycle of $\mathrm{GL}(n,q)$.
Therefore, motivated by \cite[Theorem 2.4]{Heideequiv} and according to \cite[Definition 3.5]{Heideequiv}, we say that two cyclic subspace codes $C_U$ and $C_V$ are \textbf{linearly equivalent} if there exists $i \in \{0,\ldots,n-1\}$ such that 
\[ C_U=C_{ V^{q^i}}, \]
where $V^{q^i}=\{ v^{q^i} \colon v \in V \}$.
This happens if and only if $U=\alpha V^{q^i}$, for some $\alpha \in \fqn^*$.

We can hence extend this definition to the case of cyclic subspace codes with $r$ orbits.

\begin{definition}
Let $C=\bigcup_{i \in \{1,\ldots,r\}} C_{U_i}$ and $C'=\bigcup_{i \in \{1,\ldots,r\}} C_{V_i}$, with $U_1,\ldots,U_r$ and $V_1,\ldots,V_r$ $\fq$-subspaces in $\mathcal{G}_q(n,k)$.
We say that $C$ and $C'$ are \textbf{linearly equivalent} if there exists $j \in \{0,\ldots,n-1\}$ such that
\[ C=\bigcup_{i \in \{1,\ldots,r\}} C_{V_i^{q^j}}. \]
\end{definition}

Equivalently, $C=\bigcup_{i \in \{1,\ldots,r\}} C_{U_i}$ and $C'=\bigcup_{i \in \{1,\ldots,r\}} C_{V_i}$ are linearly equivalent if there exist $\alpha_1,\ldots,\alpha_r \in \fqn^*$, $\sigma \in S_r$ and $j \in \{0,\ldots,n-1\}$ such that
\[ U_i=\alpha_i V_{\sigma(i)}^{q^j}, \]
for any $i \in \{ 1,\ldots,r \}$.

Clearly, if $C$ and $C'$ are two linearly equivalent cyclic subspace codes, then $|C|=|C'|$ and their minimum distances coincide. 
Such properties are also satisfied if we weaken the definition as follows.

\begin{definition}
Let $C=\bigcup_{i \in \{1,\ldots,r\}} C_{U_i}$ and $C'=\bigcup_{i \in \{1,\ldots,r\}} C_{V_i}$, with $U_1,\ldots,U_r$ and $V_1,\ldots,V_r$ $\fq$-subspaces in $\mathcal{G}_q(n,k)$.
We say that $C$ and $C'$ are \textbf{semilinearly equivalent} if there exists $\rho \in \mathrm{Aut}(\fqn)$ such that
\[ C=\bigcup_{i \in \{1,\ldots,r\}} C_{V_i^{\rho}}. \]
\end{definition}

Equivalently, $C=\bigcup_{i \in \{1,\ldots,r\}} C_{U_i}$ and $C'=\bigcup_{i \in \{1,\ldots,r\}} C_{V_i}$ are semilinearly equivalent if there exist $\alpha_1,\ldots,\alpha_r \in \fqn^*$, $\sigma \in S_r$ such that
\[ U_i=\alpha_i V_{\sigma(i)}^{\rho}, \]
for any $i \in \{ 1,\ldots,r \}$.

Now, we will study the equivalence issue among the family of codes $G_{n,s}$, showing that such a family contains a large number of inequivalent codes.
To this aim, we start by recalling the following technical lemma proved in \cite{NPSZ2021}.

\begin{lemma}\cite[Lemma 4.13]{NPSZ2021}\label{lem:equivbet2}
Let $n=2k$ be an even positive integer and $k\geq 3$.
Let $\xi,\eta \in \F_{q^{2k}}\setminus \F_{q^k}$.
Let $s,s'$ be positive integers such that $\gcd(s,k)=\gcd(s',k)=1$.
Let $\mu,\overline{\mu}\in \F_{q^k}^*$ and consider the $\fq$-subspaces of $\F_{q^{2k}}$
\[ W_{\mu x^{q^s},\xi}=\{ u+\xi \mu u^{q^s} \colon u \in \F_{q^k} \} \]
and
\[ W_{\overline{\mu} x^{q^{s'}},\eta}=\{ v+\eta \overline{\mu} v^{q^{s'}} \colon v \in \F_{q^k} \}. \]
There exists $\lambda \in \fqn^*$ such that
\begin{equation}\label{eq:equivalencemonomials}
    W_{\mu x^{q^s},\xi}=\lambda (W_{\overline{\mu} x^{q^{s'}},\eta})^\rho
\end{equation}
if and only if one of the following condition holds
\begin{itemize}
    \item $s\equiv s'\pmod{k}$, $B=0$, $\xi =\frac{\eta^{\rho}}{A}$ and $\overline{\mu}=\frac{\mu^{\rho^{-1}}c}{A^{\rho^{-1}}}$, where $c\in \F_{q^k}$ is such that $\N_{q^k/q}(c)=1$;
    \item $s\equiv -s'\pmod{k}$, $\xi=\frac{\eta^\rho+Aa}{A}$, $\overline{\mu}=\frac{1}{c \mu^{q^{-s}\rho^{-1}} A^{\rho^{-1}}b^{\rho^{-1}}}$ and $B=-Aa$, where $c\in \F_{q^k}$ is such that $\N_{q^k/q}(c)=1$,
\end{itemize}
where $\xi^2=a\xi+b$, $\eta^\rho=A\xi+B$ with $a,b,A,B \in \F_{q^k}$ and $\rho \in \mathrm{Aut}(\fqn)$.
\end{lemma}

Using the above result, similarly to what has been done in \cite[Theorem 4.14]{NPSZ2021}, we obtain the following result.

\begin{corollary}\label{cor:equivstronSidonpseudo}
Let $s,s'$ be two positive integers coprime with $k\geq 3$ and let $n=2k$, let $\xi,\eta \in \fqn\setminus \F_{q^k}$ and let $f_i(x)=\mu_i x^{q^s} \in \mathcal{L}_{k}$ and $g_i(x)=\overline{\mu}_i x^{q^{s'}} \in \mathcal{L}_{k}$ for $i \in \{1,\ldots,r\}$ such that $r\leq q-1$, $\N_{q^k/q}(\mu_i)\ne\N_{q^k/q}(\mu_j)$, $\N_{q^k/q}(\overline{\mu}_i)\ne\N_{q^k/q}(\overline{\mu}_j)$, $\N_{q^k/q}(\mu_i \mu_j \xi^{q^k+1})\ne 1$ and $\N_{q^k/q}(\overline{\mu}_i \overline{\mu}_j \eta^{q^k+1})\ne 1$ for every $i\ne j$.
Then $\displaystyle C=\bigcup_{i=1}^rC_{W_{f_i,\xi}}$ and $\displaystyle \overline{C}=\bigcup_{i=1}^rC_{W_{g_i,\eta}}$ are semilinearly equivalent if and only if there exist a permutation $\sigma \in S_r$ and $\rho \in \mathrm{Aut}(\fqn)$ such that, for every $i \in \{1,\ldots,r\}$, one of the following occur
\begin{itemize}
    \item $s \equiv s' \pmod{k}$, $B=0$, $\xi =\frac{\eta^{\rho}}{A}$ and $\overline{\mu}_{\sigma(i)}=\frac{\mu_i^{\rho^{-1}}c}{A^{\rho^{-1}}}$, where $c\in \F_{q^k}$ is such that $\N_{q^k/q}(c)=1$;
    \item $s \equiv -s' \pmod{k}$, $\xi=\frac{\eta^\rho+Aa}{A}$, $\overline{\mu}_{\sigma(i)}=\frac{1}{c \mu_i^{q^{-s}\rho^{-1}} A^{\rho^{-1}}b^{\rho^{-1}}}$ and $B=-Aa$, where $c\in \F_{q^k}$ is such that $\N_{q^k/q}(c)=1$,
\end{itemize}
where $\xi^2=a\xi+b$ and $\eta^\rho=A\xi+B$ with $a,b,A,B \in \F_{q^k}$.
\end{corollary}
\begin{proof}
Suppose that $C$ and $\overline{C}$ are semilinearly equivalent, that is there exist a permutation $\sigma \in S_r$, $\lambda_1,\ldots,\lambda_r \in \F_{q^n}^*$ and $\rho \in \mathrm{Aut}(\fqn)$ such that
\[ W_{f_i,\xi}=\lambda_i (W_{g_{\sigma(i)},\eta})^\rho \]
for every $i \in \{1,\ldots,r\}$.
By Lemma \ref{lem:equivbet2} then $s \equiv \pm s' \pmod{k}$.
Moreover, if $s \equiv s' \pmod{k}$ then $B=0$, $\xi =\frac{\eta^{\rho}}{A}$ and $\overline{\mu}_{\sigma(i)}=\frac{\mu_i^{\rho^{-1}}c}{A^{\rho^{-1}}}$, where $c\in \F_{q^k}$ such that $\N_{q^k/q}(c)=1$; whereas if $s \equiv -s' \pmod{k}$ then $\xi=\frac{\eta^\rho+Aa}{A}$, $\overline{\mu}_{\sigma(i)}=\frac{1}{c \mu_i^{q^{-s}\rho^{-1}} A^{\rho^{-1}}b^{\rho^{-1}}}$ and $B=-Aa$, where $c\in \F_{q^k}$ such that $\N_{q^k/q}(c)=1$, for every $i \in \{1,\ldots,r\}$.
The converse follows by Lemma \ref{lem:equivbet2}.
\end{proof}

In particular, we can give the following bound on the number of inequivalent codes of type $G_{n,s}$.

\begin{corollary}\label{cor:number}
The number of semilinearly inequivalent codes of Corollary \ref{teo:subspacecodesidon} is at least $\varphi(k)/2$, where $\varphi$ is the Euler totient function.
\end{corollary}

To prove the above corollary, we did not use the complete characterization of the equivalence given by Corollary \ref{cor:equivstronSidonpseudo}. This means that the family of the codes of form $G_{n,s}$ could be even larger.

\section{A geometric interpretation of cyclic subspace codes}

In this section we present a geometric interpretation of the Sidon property via linear sets that will give a bound on the number of orbits of a multi-orbit cyclic subspace code. 

The Sidon property in terms of linear sets reads as follows.

\begin{theorem}\label{th:characSidonLS}
Let $U$ be a $k$-dimensional $\fq$-subspace of $\fqn$. Then
$U$ is a Sidon space if and only if the only points of $L_{U\times U}\subseteq \PG(1,q^n)=\PG(\fqn\times\fqn,\fqn)$ of weight greater than one are those in $\PG(\F_q\times\fq,\fq)$. 
Furthermore, the weight of such points is $k$.

In particular, if $U$ is a Sidon space then the size of $L_{U\times U}$ is 
\[ \frac{q^k-1}{q-1}(q^k-q)+q+1. \]
\end{theorem}
\begin{proof}
Let $\alpha\in \fqn^*$.
Let $\la (1,\alpha) \ra_{\fqn}\in L_{U\times U}$. 
Then there exists some $\rho \in \fqn^*$ such that
\[ \rho(1,\alpha)\in U\times U, \]
that is $\rho \in U \cap \alpha^{-1} U$. 
Therefore, if $U$ is a Sidon space by Theorem \ref{lem:charSidon2} it follows that $\dim_{\fq}(U\cap\alpha^{-1}U)\leq 1$ if $\alpha \notin \fq$ and $\dim_{\fq}(U\cap\alpha^{-1}U)=k$ if $\alpha \in \fq$. So $w_{L_{U\times U}}(\la (1,\alpha) \ra_{\fqn})=1$ if and only if $\alpha \notin \fq$ and if $w_{L_{U\times U}}(\la (1,\alpha) \ra_{\fqn})\geq 2$ then $\alpha \in \fq$ and $w_{L_{U\times U}}(\la (1,\alpha) \ra_{\fqn})=k$.
Suppose now that the only points of $L_{U\times U}\subseteq \PG(1,q^n)$ of weight greater than one are those in $L_{U\times U}\cap \PG(\F_q\times\fq,\fq)$. Then if $\alpha \notin \fq$ we have $\dim_{\fq}(U\cap\alpha^{-1}U)\leq 1$ and so by Theorem \ref{lem:charSidon2} the subspace $U$ turns out to be a Sidon space.
The last part follows by \eqref{eq:pesicard} and \eqref{eq:pesivett}.
\end{proof}

The above result can be extended to the case of multi-Sidon spaces and, when they all have dimension $n/2$, to those cyclic subspace codes with minimum distance $n-2$.

\begin{theorem}\label{th:characmultipleSLS}
Let $\{U_1,\ldots,U_r\}$ be a set of $\fq$-subspaces in $\fqn$ and let $k_i=\dim_{\fq}(U_i)\geq 2$ for every $i \in \{1,\ldots,r\}$. Then
$\{U_1,\ldots,U_r\}$ is a multi-Sidon space if and only if 
\begin{itemize}
    \item the only points of $L_{U_i\times U_i}\subseteq \PG(1,q^n)=\PG(\fqn\times\fqn,\fqn)$ of weight greater than one are those in $ \PG(\F_q\times\fq,\fq)$, for every $i \in \{1,\ldots,r\}$;
    \item $L_{U_i\times U_i}\cap L_{U_j\times U_j}=\PG(\fq\times\fq,\fq)$, for every $i,j \in \{1,\ldots,r\}$ with $i\ne j$.
\end{itemize}
\end{theorem}
\begin{proof}
Let $\alpha\in \fqn^*$.
Let $\la (1,\alpha) \ra_{\fqn}\in L_{U_i\times U_i}\cap L_{U_j\times U_j}$, with $i\ne j$.
In particular, there exists $\rho \in \fqn^*$ such that 
\[ \rho(1,\alpha)=(u,\overline{u}), \]
for some $u,\overline{u}\in U_i$.
Hence, $\alpha=\rho^{-1}\overline{u}=\overline{u}/u$.
Similarly, $\alpha=\overline{v}/v$, for some $v,\overline{v}\in U_j$.
So that $\alpha \in U_i\cdot U_i^{-1} \cap U_j\cdot U_j^{-1}$.
The assertion now follows by Theorem \ref{th:character} and by Theorem \ref{th:characSidonLS}.
\end{proof}

We can hence derive some bounds that involve the number and the dimensions of the subspaces of a multi-Sidon space, the degree of the field extension and $q$.

\begin{theorem}\label{th:boundmulti}
Let $\{U_1,\ldots,U_r\}$ be a multi-Sidon space and let $k_i=\dim_{\fq}(U_i)\geq 2$ for every $i \in \{1,\ldots,r\}$.
Then
\[ \sum_{i=1}^r \frac{q^{k_i}-1}{q-1}(q^{k_i}-q)\leq q^n-q. \]
If $\{U_1,\ldots,U_r\}$ is a maximum multi-Sidon space then 
\[ r\leq \frac{(q^n-q)(q-1)}{(q^{n/2}-q)(q^{n/2}-1)}. \]
In particular, \[ r \leq  \begin{cases}
q-1 & \text{if } n>4,\\
q & \text{if } n=4.
\end{cases}\]
\end{theorem}
\begin{proof}
By Theorem \ref{th:characmultipleSLS}, the $L_{U_i\times U_i}$'s pairwise intersect each other only in $\PG(\fq\times\fq,\fq)$,
so that
\[ \left|\left(\bigcup_{i=1}^r L_{U_i\times U_i}\right)\setminus \PG(\fq\times\fq,\fq)\right|=\sum_{i=1}^r \frac{q^{k_i}-1}{q-1}(q^{k_i}-q).\]
Since $\left(\bigcup_{i=1}^r L_{U_i\times U_i}\right)\setminus \PG(\fq\times\fq,\fq) \subseteq \PG(1,q^n)\setminus \PG(\fq\times\fq,\fq)$, the assertion follows.
\end{proof}

Combining Theorem \ref{th:boundmulti} with the definition of multi-Sidon space we obtain the following bound on the number of orbits of a cyclic subspace code (whose subspaces have dimension $n/2$).

\begin{corollary}\label{cor:boundcyclcisubspacecodes}
Let $n=2k\geq 4$ and let $U_1,\ldots,U_r$ be $\fq$-subspaces of dimension $k$ in  $\fqn$ and let
\[ C=\bigcup_{i \in \{1,\ldots,r\}} C_{U_i} \subseteq \mathcal{G}_{q}(n,k)\]
be a subspace code. 
If the minimum distance of $C$ is $n-2$, then  
\[ r \leq  \begin{cases}
q-1 & \text{if } n>4,\\
q & \text{if } n=4.
\end{cases}\]
\end{corollary}

\begin{remark}
Corollary \ref{cor:boundcyclcisubspacecodes} implies that when $n=2k>4$ and $U_1,\ldots,U_r$ are $\fq$-subspaces of dimension $k$ in  $\fqn$ and
\[ C=\bigcup_{i \in \{1,\ldots,r\}} C_{U_i} \subseteq \mathcal{G}_{q}(n,k),\]
with $C$ of minimum distance $n-2$, then $|C|\leq q^n-1$. 
This bound cannot be derived directly from the sphere-packaging bound for subspace codes, but we note that they coincide asymptotically. 
\end{remark}

\section{Linear sets with only $r$ points of weight greater than one}

In this section we will make use of cyclic subspace codes to construct a special type of linear sets in $\PG(r-1,q^n)$ with exactly $r$ points of weight greater than one. 
More precisely, our purpose is to investigate linear sets $L_U$ in $\PG(r-1,q^n)$ of rank $k\leq (r-1)n$ containing $r$ independent points $P_1,\ldots,P_r$ such that
\[ w_{L_U}(P_1)+\ldots+w_{L_U}(P_r)=k. \]
We will characterize these linear sets, we will provide bounds on their rank and constructions, using multi-orbit cyclic subspace codes via their representatives.
Then we will give a polynomial representation of those linear sets having rank $n$.
Then we will study the dual of these linear sets, which have few intersection numbers with hyperplanes, yielding linear rank metric codes with few weights.

Linear sets of this form in $\PG(1,q^n)$ have been already studied in \cite{NPSZ2021} (see also \cite{JenaVdV,NPSZ2022}), where the authors study linear sets on the projective line containing two points whose sum of their weights equals the rank of the linear set.

We start our investigation by proving that, up to equivalence, one can write the linear sets as the one defined by the Cartesian product of $r$ $\fq$-subspaces of $\fqn$, as shown in the next result.

\begin{proposition}\label{prop:structure}
Let $L_W$ be an $\fq$-linear set in $\mathrm{PG}(r-1,q^n)$ of rank $k$ containing $r$ independent points $P_1=\langle \mathbf{v}_1\rangle_{\fqn},\ldots,P_r=\langle \mathbf{v}_r\rangle_{\fqn}$ such that
\[ w_{L_W}(P_1)+\ldots+w_{L_W}(P_r)=k\leq (r-1)n. \] 
Then $L_W$ is $\mathrm{PGL}(r,q^n)$-equivalent to $L_U$ where 
$U=U_1\times \ldots \times U_r,$
for some $\fq$-subspaces $U_1,\ldots,U_r$ of $\fqn$ such that $\dim_{\fq}(U_i)=w_{L_W}(P_i)$ for every $i \in \{1,\ldots,r\}$.
Moreover, if $k\leq n$ then $U_1,\ldots,U_r$ can be chosen in such a way that $U_1+\ldots+U_r$ is a direct sum.
\end{proposition}
\begin{proof}
Consider a collineation $\varphi \in \mathrm{PGL}(r,q^n)$ such that
\[ \varphi(P_i)=\langle \mathbf{e}_i \rangle_{\fqn}, \]
$\mathbf{e}_i\in \fqn^r$ is the vector whose $i$-th component is $1$ and all the others are zero.
Denote by $f$ the associated linear isomorphism of $\fqn^r$ and let $U=f(W)$. 
Since 
\[ W= (W\cap \langle \mathbf{v}_1\rangle_{\fqn})\oplus \ldots \oplus (W\cap \langle \mathbf{v}_r\rangle_{\fqn}), \]
then
\[ U= (U\cap \langle \mathbf{e}_1\rangle_{\fqn})\oplus \ldots \oplus (U\cap \langle \mathbf{e}_r\rangle_{\fqn}). \]
Therefore, $U\cap \langle \mathbf{e}_i\rangle_{\fqn}$ is equal to $\{(0,\ldots,0,u,0,\ldots,0)\colon u \in U_i\}$ for any $i \in \{1,\ldots,r\}$ and for a subspace $U_i$ of $\fqn$. So we obtain the first part of the assertion.
Now suppose that $k\leq n$. Our aim is to prove that there exists $(\lambda_1,\ldots,\lambda_r)\in (\fqn^*)^r$ such that $\lambda_i U_i\cap \langle \lambda_j U_j \colon j \in\{1,\ldots,r\}\setminus\{i\} \rangle_{\fq}=\{\mathbf{0}\}$ for any $i\in \{1,\ldots,r\}$.
We proceed by finite induction proving that there exist $\lambda_i \in \fqn^*$ such that
\begin{equation}\label{eq:prove}
\lambda_i U_i\cap \langle \lambda_j U_j \colon j \in\{1,\ldots,i-1\} \rangle_{\fq}=\{\mathbf{0}\},
\end{equation}
for every $i \in \{2,\ldots,r\}$.
This is enough to show the last part of the assertion. The case $i=2$ has been proved in \cite[Proposition 3.2]{NPSZ2021}. 
Now, let $i\in \{1,\ldots,r\}$ and suppose that \eqref{eq:prove} holds for every $j\leq i$. 
Then, using \eqref{eq:prove} recursively we get
\[ \dim_{\fq}(\langle \lambda_j U_j \colon j \in \{1,\ldots,i\} \rangle_{\fq})= \dim_{\fq} (\lambda_i U_i)+\dim_{\fq}(\langle \lambda_j U_j \colon j \in \{1,\ldots,i-1\} \rangle_{\fq})\]
\[-\dim_{\fq}(\lambda_i U_i \cap \langle \lambda_j U_j \colon j \in \{1,\ldots,i-1\} \rangle_{\fq})=\dim_{\fq}(U_1)+\ldots+\dim_{\fq}(U_i). \]
Let $h \in \{1,\ldots,i\}$, then
\[ \dim_{\fq}(\lambda_h U_h \cap \langle \lambda_j U_j \colon j \in\{1,\ldots,i\}\setminus\{h\} \rangle_{\fq})=\dim_{\fq}(U_h)\]
\[+\dim_{\fq}(\langle \lambda_j U_j \colon j \in\{1,\ldots,i\}\setminus\{h\} \rangle_{\fq})-\dim_{\fq}(\langle \lambda_j U_j \colon j \in\{1,\ldots,i\}\rangle_{\fq}). \]
The last two equations imply $\dim_{\fq}(\lambda_h U_h \cap \langle \lambda_j U_j \colon j \in\{1,\ldots,i\}\setminus\{h\}\rangle_{\fq})=0$.

We now prove \eqref{eq:prove}.
Suppose now that 
\[\lambda_i U_i\cap \langle \lambda_j U_j \colon j \in\{1,\ldots,i-1\} \rangle_{\fq}=\{\mathbf{0}\}.\]
Suppose that for each $\lambda \in \fqn^*$
\begin{equation}\label{eq:linsetline} \lambda U_{i+1} \cap \langle \lambda_j U_j \colon j \in\{1,\ldots,i\} \rangle_{\fq}\ne \{\mathbf{0}\}.
\end{equation}
Denote by $\overline{U}=\langle \lambda_j U_j \colon j \in\{1,\ldots,i\} \rangle_{\fq}$. By \eqref{eq:linsetline} it follows that the $\fq$-linear set $L_{U_{i+1}\times \overline{U}}$ of rank less than or equal to $n$ contains at least $q^n+1$ points, a contradiction to \eqref{eq:card}. So, there exists $\lambda_{i+1} \in \fqn^*$ such that $\lambda_{i+1}U_{i+1}\cap \overline{U}=\{\mathbf{0}\}$.
\end{proof}

In terms of linear sets, Theorem \ref{th:character} reads as a characterization of linear sets with only $r$ points of weight greater than one (in independent position).

\begin{theorem}\label{th:SS-1}
Let $U_1,\ldots,U_r$ be $r$ $\fq$-subspaces of $\fqn$ such that $\sum_{i=1}^r \dim_{\fq}(U_i)=k\leq (r-1)n$.
Let $k_i=\dim_{\fq}(U_i)\geq 2$,  for any $i \in \{1,\ldots,r\}$ and $U=U_1\times \ldots \times U_r$.
Then, the following are equivalent
\begin{itemize}
    \item[i)] all the points of $L_U$ different from $P_i=\langle\mathbf{e}_i\rangle_{\fqn}$ with $i\in\{1,\ldots,r\}$ have weight one in $L_U$;
    \item[ii)] $U_i\cdot U_i^{-1} \cap U_j\cdot U_j^{-1}=\fq$ for any $i,j \in \{1,\ldots,r\}$ with $i\ne j$;
    \item[iii)] $\dim_{\fq}(U_i\cap \alpha U_j)\leq 1$, for every $\alpha \in \fqn^*$ and $i,j \in \{1,\ldots,r\}$ with $i\ne j$.
\end{itemize}
\end{theorem}

Now, we will prove a bound on the rank of these linear sets. This bound relies on the following result from \cite{NPSZ2021}.

\begin{theorem}\cite[Theorem 4.4]{NPSZ2021}\label{th:line2points}
Let $L$ be an $\fq$-linear set of $\PG(1,q^n)$ with exactly two points $P$ and $Q$ of weight greater than one, then 
\[ w_L(P)\leq \frac{n}2 \,\,\text{and}\,\,w_L(Q)\leq \frac{n}2. \]
\end{theorem}

The next result extends \cite[Theorem 4.4]{NPSZ2021} to linear sets in $\mathrm{PG}(r-1,q^n)$.

\begin{theorem}\label{th:boundLS}
Let $L_W$ be an $\fq$-linear set in $\mathrm{PG}(r-1,q^n)$ of rank $k$ containing $r$ independent points $P_1=\langle \mathbf{v}_1\rangle_{\fqn},\ldots,P_r=\langle \mathbf{v}_r\rangle_{\fqn}$ such that
\[ w_{L_W}(P_1)+\ldots+w_{L_W}(P_r)=k\leq (r-1)n. \] 
Suppose that $w_{L_W}(Q)\leq 1$ for every $Q \in \mathrm{PG}(r-1,q^n) \setminus\{P_1,\ldots,P_r\}$.
Then
\begin{equation}\label{eq:weightpoints} w_{L_W}(P_i)\leq \frac{n}2 
\end{equation}
for every $i\in \{1,\ldots,r\}$.
Moreover,
\begin{itemize}
    \item $\mathrm{Rank}(L_W)\leq \frac{rn}2$;
    \item if $\mathrm{Rank}(L_W)= \frac{rn}2$, then $w_{L_W}(P_i)=\frac{n}2$ for every $i \in \{1,\ldots,r\}$.
\end{itemize}
\end{theorem}
\begin{proof}
By Proposition \ref{prop:structure} $L_W$ is $\mathrm{PGL}(r,q^n)$-equivalent to $L_U$ where 
$U=U_1\times \ldots \times U_r,$
for some $\fq$-subspaces $U_1,\ldots,U_r$ of $\fqn$ such that $\dim_{\fq}(U_i)=w_{L_W}(P_i)$ for every $i \in \{1,\ldots,r\}$.
So, the only points with weight greater than one in $L_U$ are $Q_i=\langle \mathbf{e}_i \rangle_{\fqn}$ for every $i \in \{1,\ldots,r\}$ and $w_{L_W}(P_i)=w_{L_U}(Q_i)$ for every $i \in \{1,\ldots,r\}$.
Let $i,j \in \{1,\ldots,r\}$ with $i\ne j$ and consider the line $\ell=\mathrm{PG}(\langle \mathbf{e}_i,\mathbf{e}_j\rangle_{\fqn},\fqn)=\langle Q_i,Q_j \rangle$.
We first observe that for each $Q \in \ell$ then $w_{L_U}(Q)=w_{L_U\cap \ell}(Q)$ and hence the only points of weight greater than one in $L_U\cap \ell$ are $Q_i$ and $Q_j$.
By Theorem \ref{th:line2points} we have that
\[ w_{L_U\cap \ell}(Q_i)=w_{L_U}(Q_i)\leq \frac{n}2\,\,\,\text{and}\,\,\,w_{L_U\cap \ell}(Q_j)=w_{L_U}(Q_j)\leq \frac{n}2, \]
and the first part of the assertion then follows.
The second part is a consequence of \eqref{eq:weightpoints}.
\end{proof}

\begin{remark}\label{rk:sizemaximumls}
Let $L_W$ be an $\fq$-linear set in $\mathrm{PG}(r-1,q^n)$ of rank $\frac{rn}2$ containing $r$ independent points $P_1=\langle \mathbf{v}_1\rangle_{\fqn},\ldots,P_r=\langle \mathbf{v}_r\rangle_{\fqn}$ such that $w_{L_W}(P_i)=\frac{n}2$ for every $i \in \{1,\ldots,r\}$.
By \eqref{eq:pesicard} and \eqref{eq:pesivett}, recalling that $N_i$ is the number of points in $L_W$ having weight $i$, we obtain that
\[ N_1=q^{\frac{rn}2-1}+\ldots+q^{\frac{n}2}-(r-1)(q^{\frac{n}2-1}+\ldots+1),\,\,\, N_{\frac{n}2}=r\,\,\text{and}\,\, N_0=\frac{q^{rn}-1}{q^n-1}-N_1-N_{\frac{n}2}. \]
In particular
\[ |L_W|=q^{\frac{rn}2-1}+\ldots+q^{\frac{n}2}-(r-1)(q^{\frac{n}2-1}+\ldots+q)+1. \]
\end{remark}

We can now construct linear sets in $\mathrm{PG}(r-1,q^n)$ with exactly $r$ points of weight greater than one and in particular we will see constructions when then rank is maximum, that is $\frac{rn}2$.
Multi-Sidon spaces can be used to this aim.

\begin{proposition}
Let $U_1,\ldots,U_r$ be $r$ $\fq$-subspaces of $\fqn$ of dimensions $k_1,\ldots,k_r$, respectively.
Let $U=U_1\times \ldots \times U_r$.
If $\{U_1,\ldots,U_r\}$ is a multi-Sidon space, then $L_U$ is an $\fq$-linear set of rank $k_1+\ldots+k_r$ in $\PG(r-1,q^n)$ such that the only points of weight greater than one are the $P_i=\langle \mathbf{e}_i\rangle_{\fqn}$'s with $i \in \{1,\ldots,r\}$.
\end{proposition}
\begin{proof}
Since $\{U_1,\ldots,U_r\}$ is a multi-Sidon space, we have that
\[ \dim_{\fq}(U_i\cap \alpha U_j)\leq 1, \]
for every $\alpha \in \fqn$ and $i,j \in \{1,\ldots,r\}$ with $i\ne j$.
The assertion then follows directly by Theorem \ref{th:SS-1}.
\end{proof}

In particular, when considering maximum multi-Sidon spaces, we can get examples of linear sets whose rank satisfy the equality in Theorem \ref{th:boundLS}.

\begin{corollary}\label{cor:codelinearsets}
Let $n=2k$ and $U_1,\ldots,U_r$ be $r$ $\fq$-subspaces of $\fqn$ of dimension $k$.
Let $U=U_1\times \ldots \times U_r$.
If $\{U_1,\ldots,U_r\}$ is a multi-Sidon space, then $L_U$ is an $\fq$-linear set of rank $rk$ in $\PG(r-1,q^n)$ such that the only points of weight greater than one are the $P_i=\langle \mathbf{e}_i\rangle_{\fqn}$'s with $i \in \{1,\ldots,r\}$.
\end{corollary}

The multi-cyclic subspace codes found in Proposition \ref{prop:examplepseudo} can be used to get constructions of such linear sets via Corollary \ref{cor:codelinearsets}.

\begin{corollary}\label{cor:ex1ls}
Let $s$ be a positive integer coprime with $k\geq 2$ and $n=2k$, let $\xi \in \fqn\setminus \F_{q^k}$ and let $f_i(x)=\mu_i x^{q^s} \in \mathcal{L}_{k}$ for $i \in \{1,\ldots,r\}$ such that $r\leq q-1$,  $\N_{q^k/q}(\mu_i)\ne\N_{q^k/q}(\mu_j)$ and $\N_{q^k/q}(\mu_i \mu_j \xi^{q^k+1})\ne 1$ for every $i\ne j$.
Denote by $U_i=W_{f_i,\xi}$ for every $i \in \{1,\ldots,r\}$.
Let $U=U_1\times \ldots \times U_r$,
then $L_U$ is an $\fq$-linear sets of rank $rk$ in $\PG(r-1,q^n)$ such that the only points of weight greater than one are the $P_i=\langle \mathbf{e}_i\rangle_{\fqn}$'s with $i \in \{1,\ldots,r\}$.
\end{corollary}

Another construction may be obtained by extending a multi-Sidon space with subfield $\F_{q^k}$ of $\fqn$.

\begin{corollary}\label{cor:ex2ls}
Let $s$ be a positive integer coprime with $k\geq 2$ and $n=2k$, let $\xi \in \fqn\setminus \F_{q^k}$ and let $f_i(x)=\mu_i x^{q^s} \in \mathcal{L}_{k}$ for $i \in \{1,\ldots,r-1\}$ such that $r\leq q-1$, $\N_{q^k/q}(\mu_i)\ne\N_{q^k/q}(\mu_j)$ and $\N_{q^k/q}(\mu_i \mu_j \xi^{q^k+1})\ne 1$ for every $i\ne j$.
Denote by $U_i=W_{f_i,\xi}$ for every $i \in \{1,\ldots,r-1\}$ and by $U_r=\F_{q^k}$.
Let $U=U_1\times \ldots \times U_r$,
then $L_U$ is an $\fq$-linear set of rank $rk$ in $\PG(r-1,q^n)$ such that the only points of weight greater than one are the $P_i=\langle \mathbf{e}_i\rangle_{\fqn}$'s with $i \in \{1,\ldots,r\}$.
\end{corollary}
\begin{proof}
By Theorem \ref{th:SS-1} and by Proposition \ref{prop:examplepseudo}, it is enough to show that
\[ \dim_{\fq}(\F_{q^k}\cap\alpha U_i)\leq 1, \]
for every $i \in \{1,\ldots,r-1\}$ for any $\alpha \in \fqn^*$.
Let $\alpha \in \fqn^*$.
Then there exist $\alpha_0,\alpha_1 \in \F_{q^k}$ such that $\alpha=\alpha_0+\alpha_1 \xi$. 
Let $v \in \F_{q^k}$, we look for solutions in $u \in \F_{q^k}$ such that
\[ v=\alpha(u+u^{q^s}\mu_i\xi). \]
Using that $\xi^2=a\xi+b$, for some $a,b \in \F_{q^k}$, and that $\alpha=\alpha_0+\alpha_1 \xi$, the above relation reads as follows
\[ 
\left\{
\begin{array}{ll}
\alpha_0 u +b \alpha_1 u^{q^s} \mu_i=v,\\
\alpha_1 u+\alpha_0 \mu_i u^{q^s} +\mu_i\alpha_1 a u^{q^s}=0.
\end{array}
\right.
\]
Since $\gcd(s,k)=1$, $\alpha_1 u+\alpha_0 \mu_i u^{q^s} +\mu_i\alpha_1 a u^{q^s}$ can be seen as a non-zero $q^s$-polynomial over $\F_{q^k}$ of $q^s$-degree at most one. Therefore, by Theorem \ref{Gow} it follows that the number of $u \in \F_{q^k}$ which are solutions of the above system is at most $q$ and hence 
\[ \dim_{\fq}(\F_{q^k}\cap \alpha U_i)\leq 1. \]
\end{proof}

In the next subsection we will prove that the two $\fq$-subspaces defined in the above corollaries are $\Gamma\mathrm{L}(r,q^n)$-inequivalent.

\subsection{$\Gamma\mathrm{L}(r,q^n)$-equivalence}

We now provide a useful tool that can be used to study the $\Gamma\mathrm{L}(r,q^n)$-equivalence of the subspaces representing linear sets with only $r$ points of weight greater than one, which extends \cite[Lemma 3.3]{NPSZ2021}. 

\begin{theorem}\label{prop:criteria}
Let $U_1,\ldots,U_r$ and $W_1,\ldots,W_r$ be $\fq$-subspaces of $\fqn$ such that
\[ \dim_{\fq}(U_1)+\ldots+ \dim_{\fq}(U_r)=k\leq(r-1)n \]
and 
\[ \dim_{\fq}(W_1)+\ldots+ \dim_{\fq}(W_r)=k\leq(r-1)n. \]
Let $U=U_1\times \ldots\times U_r$ and $W=W_1\times \ldots\times W_r$ and suppose that $P_i=\langle\mathbf{e}_i\rangle_{\fqn}$ are the only points of weight greater than one in $L_U$ and in $L_W$. 
Then $U$ and $W$ are $\mathrm{\Gamma L}(r,q^n)$-equivalent if and only if there exist a permutation $\sigma \in S_r$, $\lambda_1,\ldots,\lambda_r \in \fqn^*$ and $\rho \in \mathrm{Aut}(\fqn)$ such that
\[ W_i=\lambda_iU_{\sigma(i)}^\rho, \]
for every $i\in \{1,\ldots,r\}$.
In particular, $\dim_{\fq}(W_i)=\dim_{\fq}(U_{\sigma(i)})$, for every $i\in \{1,\ldots,r\}$.
\end{theorem}
\begin{proof}
The ``if'' part is trivial.
So, suppose that $U$ and $W$ are $\Gamma\mathrm{L}(r,q^n)$-equivalent via $\varphi \in \mathrm{\Gamma L}(r,q^n)$ defined by the matrix $A\in \mathrm{GL}(r,q^n)$ and the automorphism $\rho \in \mathrm{Aut}(\fqn)$.
Let $\dim_{\fq}(U_i)=k_i$ and $\dim_{\fq}(W_i)=k_i'$, for any $i \in \{1,\ldots,r\}$.

By Theorem \ref{th:character}, we have that
    \begin{itemize}
        \item[a)] $\dim_{\fq}(U\cap \langle \mathbf{v}\rangle_{\fqn})\leq 1$ for any $\mathbf{v}\in \fqn^r$ such that $\langle\mathbf{v}\rangle_{\fqn}\notin \{\langle \mathbf{e}_i \rangle_{\fqn} \colon i \in \{1,\ldots,r\}\}$;
        \item[b)] $\dim_{\fq}(U\cap \langle \mathbf{e}_i\rangle_{\fqn})=k_i$ for any $i \in \{1,\ldots,r\}$, 
    \end{itemize}
and the same properties hold for $W$.
As a consequence, there exists $\sigma \in S_r$ such that $\varphi(\langle \mathbf{e}_i\rangle_{\fqn})=\langle \mathbf{e}_{\sigma(i)}\rangle_{\fqn}$.
In particular, this means that $A$ is the product of a permutation matrix (the one induced by $\sigma$) and a diagonal matrix (whose diagonal elements are $\lambda_1,\ldots,\lambda_r\in \fqn^*$).
So,
\[ W_i=\lambda_iU_{\sigma(i)}^\rho, \]
for every $i\in \{1,\ldots,r\}$ and  $\{k_1,\ldots,k_r\}=\{k_1',\ldots,k_r'\}$.
\end{proof}

As an easy consequence we obtain that the dimensions of the subspaces are not in one-to-one correspondence, then the two subspaces are inequivalent.

\begin{corollary}
Let $U_1,\ldots,U_r$ and $W_1,\ldots,W_r$ be $\fq$-subspaces of $\fqn$ of dimension $k_1,\ldots,k_r$ and $k_1',\ldots,k_r'$, respectively.
Let $U=U_1\times \ldots\times U_r$ and $W=W_1\times \ldots\times W_r$ and suppose that $P_i=\langle\mathbf{e}_i\rangle_{\fqn}$ are the only points of weight greater than one in $L_U$ and in $L_W$. 
If $(k_1,\ldots,k_r)$ cannot be obtained from $(k_1',\ldots,k_r')$ by permuting their entries, then $U_1\times\ldots\times U_r$ and $W_1\times \ldots\times W_r$ are $\Gamma\mathrm{L}(r,q^n)$-inequivalent.
\end{corollary}

The $\Gamma\mathrm{L}(r,q^n)$-equivalence of the examples in Proposition \ref{prop:examplepseudo} has been already studied in Corollary \ref{cor:equivstronSidonpseudo}.
Whereas, as shown in the next result, the examples of Corollaries \ref{cor:ex1ls} and \ref{cor:ex2ls} cannot be $\Gamma\mathrm{L}(r,q^n)$-equivalent.

\begin{proposition}\label{prop:inequivex1and2}
Let $s,s'$ be two positive integers coprime with $k\geq 2$ and let $n=2k$, let $\xi,\eta \in \fqn\setminus \F_{q^k}$ and let $f_i(x)=\mu_i x^{q^s} \in \mathcal{L}_{k}$ and $g_\ell(x)=\overline{\mu}_\ell x^{q^{s'}} \in \mathcal{L}_{k}$ for $i \in \{1,\ldots,r\}$ and $\ell \in \{1,\ldots,r-1\}$ such that $r\leq q-1$, $\N_{q^k/q}(\mu_i)\ne\N_{q^k/q}(\mu_j)$, $\N_{q^k/q}(\overline{\mu}_i)\ne\N_{q^k/q}(\overline{\mu}_j)$, $\N_{q^k/q}(\mu_i \mu_j \xi^{q^k+1})\ne 1$ and $\N_{q^k/q}(\overline{\mu}_i \overline{\mu}_j \eta^{q^k+1})\ne 1$ for every $i\ne j$.
Then $U=W_{f_1,\xi}\times \ldots\times W_{f_r,\xi}$ and $W=W_{g_1,\eta}\times\ldots\times W_{g_{r-1},\eta}\times\F_{q^k}$ are $\Gamma\mathrm{L}(r,q^n)$-inequivalent.
\end{proposition}
\begin{proof}
By Proposition \ref{prop:criteria}, if $U$ and $W$ are $\Gamma\mathrm{L}(r,q^n)$-equivalent then there exist $i \in \{1,\ldots,r\}$, $\lambda \in \fqn^*$ and $\rho \in \mathrm{Aut}(\fqn)$ such that
\[ \F_{q^k}=\lambda ( W_{f_i,\xi} )^\rho. \]
This cannot happen since the $\fq$-subspace $W_{f_i,\xi}$ has the property that
\[ \dim_{\fq}(W_{f_i,\xi}\cap \alpha W_{f_i,\xi})\leq 1, \]
for every $\alpha \in \fqn\setminus \fq$.
This property is preserved also by $\lambda ( W_{f_i,\xi} )^\rho$, but clearly $\dim_{\fq}(\F_{q^k}\cap \alpha \F_{q^k})=k$ if $\alpha \in \F_{q^k}\setminus \fq$.
\end{proof}

Moreover, arguing as before, the $\Gamma\mathrm{L}(r,q^n)$-equivalence between two examples in Corollary \ref{cor:ex2ls} is determined in the next result.

\begin{proposition}\label{prop:equiv2exsubfield}
Let $s,s'$ be two positive integers coprime with $k$ and let $n=2k$, let $\xi,\eta \in \fqn\setminus \F_{q^k}$ and let $f_i(x)=\mu_i x^{q^s} \in \mathcal{L}_{k}$ and $g_i(x)=\overline{\mu}_i x^{q^{s'}} \in \mathcal{L}_{k}$ for $i \in \{1,\ldots,r-1\}$ such that $r\leq q-1$, $\N_{q^k/q}(\mu_i)\ne\N_{q^k/q}(\mu_j)$, $\N_{q^k/q}(\overline{\mu}_i)\ne\N_{q^k/q}(\overline{\mu}_j)$, $\N_{q^k/q}(\mu_i \mu_j \xi^{q^k+1})\ne 1$ and $\N_{q^k/q}(\overline{\mu}_i \overline{\mu}_j \eta^{q^k+1})\ne 1$ for every $i\ne j$.
Then $U=W_{f_1,\xi}\times \ldots\times W_{f_{r-1},\xi}\times \F_{q^k}$ and $W=W_{g_1,\eta}\times\ldots\times W_{g_{r-1},\eta}\times\F_{q^k}$ are $\Gamma\mathrm{L}(r,q^n)$-equivalent if and only if there exist a permutation $\sigma \in S_{r-1}$, $\lambda_1,\ldots,\lambda_{r-1} \in \fqn^*$ and $\rho \in \mathrm{Aut}(\fqn)$ such that, for every $i \in \{1,\ldots,r-1\}$, one of the following occurs
\begin{itemize}
    \item $s \equiv s' \pmod{k}$, $B=0$, $\xi =\frac{\eta^{\rho}}{A}$ and $\overline{\mu}_{\sigma(i)}=\frac{\mu_i^{\rho^{-1}}c}{A^{\rho^{-1}}}$, where $c\in \F_{q^k}$ is such that $\N_{q^k/q}(c)=1$;
    \item $s \equiv -s' \pmod{k}$, $\xi=\frac{\eta^\rho+Aa}{A}$, $\overline{\mu}_{\sigma(i)}=\frac{1}{c \mu_i^{q^{-s}\rho^{-1}} A^{\rho^{-1}}b^{\rho^{-1}}}$ and $B=-Aa$, where $c\in \F_{q^k}$ is such that $\N_{q^k/q}(c)=1$,
\end{itemize}
where $\xi^2=a\xi+b$ and $\eta^\rho=A\xi+B$ with $a,b,A,B \in \F_{q^k}$.
\end{proposition}

\subsection{Dual linear sets}

Interestingly, the linear sets $L_U$ of this section by duality give examples of sets with only three weights with respect to the hyperplanes and for each of these values there exists at least one hyperplane having this weight with respect to $L_U$.

\begin{theorem}\label{th:3weighthyper}
Let $L_U$ be an $\fq$-linear set in $\mathrm{PG}(r-1,q^n)$ of rank $\frac{rn}2$ containing $r$ independent points $P_1=\langle \mathbf{v}_1\rangle_{\fqn},\ldots,P_r=\langle \mathbf{v}_r\rangle_{\fqn}$ such that
\[ w_{L_U}(P_1)+\ldots+w_{L_U}(P_r)=\frac{rn}2. \] 
Then for every hyperplane $H$ in $\mathrm{PG}(r-1,q^n)$
\[ w_{L_U^\tau}(H)\in \left\{ \frac{rn}2-n,\frac{rn}2-n+1,\frac{rn}2-n+r \right\}. \]
More precisely,
\begin{itemize}
    \item $w_{L_U^\tau}(H)=\frac{rn}2-n+r$ if and only if $H=P_i^\tau$ for some $i \in \{1,\ldots,r\}$;
    \item $w_{L_U^\tau}(H)=\frac{rn}2-n+1$ if and only if $H=P^\tau$ for some $P \in L_U\setminus\{P_1,\ldots,P_r\}$;
    \item $w_{L_U^\tau}(H)=\frac{rn}2-n$ if and only if $H=P^\tau$ for some $P \in \PG(r-1,q^n)\setminus L_U$.
\end{itemize}
Moreover, 
\begin{itemize}
    \item the number of hyperplanes $H$ such that $w_{L_U^\tau}(H)=\frac{rn}2-n+r$ is $r$;
    \item the number of hyperplanes $H$ such that $w_{L_U^\tau}(H)=\frac{rn}2-n+1$ is $q^{\frac{rn}2-1}+\ldots+q^{\frac{n}2}-(r-1)(q^{\frac{n}2-1}+\ldots+1)$;
    \item the number of hyperplanes $H$ such that $w_{L_U^\tau}(H)=\frac{rn}2-n$ is $\frac{q^{rn}-1}{q^n-1}-q^{\frac{rn}2-1}-\ldots-q^{\frac{n}2}+(r-1)(q^{\frac{n}2-1}+\ldots+q)-1$.
\end{itemize}
\end{theorem}
\begin{proof}
By Theorem \ref{th:boundLS}, it follows that all the point $P_i$ have the same weight, that is, $w_{L_U}(P_i)=\frac{n}2$ for every $i$.
By applying Proposition \ref{prop:weightdual} choosing $H$ to be a hyperplane, we obtain
\[ w_{L_U^\tau}(H)=w_{L_U}(H^\tau)+\frac{rn}2-(s+1)n. \]
The assertion then follows by Remark \ref{rk:sizemaximumls}.
\end{proof}

The interest for these linear sets with few possible values for the weight of the hyperplanes is especially due to their coding theoretical counterparts that now we will briefly describe. 

Let $U$ be an $\fq$-subspace of $\F_{q^m}^r$ such that $\langle U \rangle_{\F_{q^m}}=\F_{q^m}^r$.
Let $G$ be a matrix in $\F_{q^m}^{r\times \frac{rn}2}$ whose columns form an $\fq$-basis of $U$. 
Let denote by $\mathcal{C}$ be the $\F_{q^m}$-span of the rows of $G$ and equip it with the \emph{rank metric}, which is defined as follows
\[ d(\mathbf{u},\mathbf{v})=w(\mathbf{u}-\mathbf{v})=\dim_{\fq}(\langle u_1-v_1,\ldots,u_r-v_r\rangle_{\fq}), \]
for any $\mathbf{u}=(u_1,\ldots,u_r), \mathbf{v}=(v_1,\ldots,v_r) \in \F_{q^m}^r$.
So, $\mathcal{C}$ is a \textbf{rank metric code} in $\mathbb{F}_{q^m}^r$ and $w(\mathbf{u})$ is said the \textbf{rank weight} of $\mathbf{u}$; see \cite{surelisa,surjohn} for more details on rank metric codes.
Noting that any element of $\mathcal{C}$ may be seen as $\mathbf{x}G$ for some $\mathbf{x}\in \mathbb{F}_{q^m}^r$, then in \cite{Ra} (see also \cite{Alfarano2021}) it was shown that
\[ w(\mathbf{x}G)=\frac{rn}2-\dim_{\fq}(U \cap \mathbf{x}^\perp), \]
where $\mathbf{x}^\perp$ is the hyperplane whose equation is defined by the entries of $\mathbf{x}$.

As a consequence of Theorem \ref{th:3weighthyper} and the above described connection, the rank metric codes $\mathcal{C}$ have exactly three nonzero weights, which are $n-r,n-1,n$. 
In particular, they are examples of $(r-1)$-\textbf{almost MRD codes}, see \cite{delaFFA}, which means that the rank defect in the Singleton bound for rank metric codes is $r-1$.
Furthermore, we can also derive their weight distribution, indeed except for the zero vector we have
\begin{itemize}
    \item $(q^n-1)r$ elements in $\mathcal{C}$ having rank weight $n-r$;
    \item $(q^n-1)\left(q^{\frac{rn}2-1}+\ldots+q^{\frac{n}2}-(r-1)(q^{\frac{n}2-1}+\ldots+1)\right)$ elements in $\mathcal{C}$ having rank weight $n-1$;
    $(q^n-1)\left(\frac{q^{rn}-1}{q^n-1}-q^{\frac{rn}2-1}-\ldots-q^{\frac{n}2}+(r-1)(q^{\frac{n}2-1}+\ldots+q)-1\right)$ elements in $\mathcal{C}$ having rank weight $n$;
\end{itemize}

Moreover, using \cite[Theorem 4.8]{Alfarano2021}, from $L_U$ we can also construct linear Hamming metric codes with only three weights and for which we can completely establish their weight distribution, as already done for some classes of linear sets (see also \cite{NapZullo,ZiniZulloScatt}). 

\section{Conclusions}

We conclude the paper with some problems that we think could be of interest for the reader.

\begin{open}
In Theorem \ref{th:canonicalform}, we find a canonical form and explicit conditions for multi-orbit subspace codes with large minimum distance. Can this be used to provide new constructions and/or classification results?
\end{open}

\begin{open}
In Corollary \ref{cor:number}, we provide a lower bound on the number of semilinearly inequivalent codes of the form $G_{n,s}$. Determine the precise number of inequivalent codes of the form $G_{n,s}$.
\end{open}

Moreover, a multi-Sidon space $U$ still preserves the uniqueness (up to some scalars in $\fq$) of the factorization of the product of any two elements contained in $U$.

\begin{proposition}\label{prop:strongweakSidon}
Let $\{U_1,\ldots,U_r\}$ be a multi-Sidon space. 
Let $k_i=\dim_{\fq}(U_i)\geq 2$ for any $i \in \{1,\ldots,r\}$. 
Then for any $i,j \in \{1,\ldots,r\}$ if $a,c \in U_i$ and $b,d \in U_j$ are such that $ab=cd$ then
\[ \{a\fq, b\fq\}=\{c\fq, d\fq\}. \]
\end{proposition}
\begin{proof}
By Corollary \ref{cor:UiUj}, when $i\ne j$ the $\fq$-subspaces $U_i$ and $U_j$ satisfy the property that for any nonzero $a,c \in U_i$ and nonzero $b,d \in U_j$ the equality $ab=cd$ implies that $a\fq=c\fq$ and $b\fq=d\fq$.
The condition $\dim_{\fq}(U_i\cap\alpha U_i)\leq 1$ for every $\alpha \in \fqn\setminus \fq$ and $i \in \{1,\ldots,r\}$ is equivalent to the fact that $U_i$ is a Sidon space for every $i$ because of Theorem \ref{lem:charSidon2}.
Therefore the assertion follows.
\end{proof}

So, a possible generalization of Sidon spaces could be the following one.

Consider $U_1,\ldots,U_r$ be $r$ distinct $\fq$-subspaces of $\fqn$. We say that $\{U_1,\ldots,U_r\}$ is a \textbf{weak multi-Sidon space} of $\fqn$ if the product of any two nonzero elements from $U_i$'s may be uniquely factorized, up to a scalar factor from $\fq$. More precisely, for any $i,j \in \{1,\ldots,r\}$ if $a,c \in U_i$ and $b,d \in U_j$ are such that $ab=cd$ then
\[ \{a\fq, b\fq\}=\{c\fq, d\fq\}. \]

This latter definition generalizes the property of being multi-Sidon space, as proved in Proposition \ref{prop:strongweakSidon}.
However, we still do not know whether or not there exists examples of weakly multi-Sidon spaces which are not multi-Sidon spaces. 
So, we propose the following problem.

\begin{problem}
Are the properties of being multi-Sidon space and weaker multi-Sidon space equivalent?
If not, construct counterexamples.
\end{problem}

\section*{Acknowledgements}

I would like to thank Oriol Serra, who suggested me to study Sidon spaces during the conference \emph{Discretaly} (2018). I would also like to thank the anonymous referee for their suggestions.
The research was supported by the project ``VALERE: VanviteLli pEr la RicErca" of the University of Campania ``Luigi Vanvitelli'' and was partially supported by the Italian National Group for Algebraic and Geometric Structures and their Applications (GNSAGA - INdAM).

\newpage

\appendix

\section{Polynomial representation of linear sets in $\mathrm{PG}(r-1,q^n)$ determined by $r$ points}

One of the most studied research topic on polynomials over finite fields regards the construction of polynomials or functions for which their value sets over a fixed finite field has some interesting properties, such as being small, large or with other preassigned conditions. Clearly, the problem becomes more and more difficult when considering $t$-tuple of polynomials.
In this appendix, we will show examples of $t$-tuples of linearized polynomials such that the size of the value set
\[ \mathrm{Im}\left( \frac{f_1(x)}x,\ldots,\frac{f_t(x)}x\right)=\left\{ \left(\frac{f_1(x)}x,\ldots, \frac{f_t(x)}x\right) \colon x \in \fqn^* \right\} \]
is determined, see \cite{Ball,BBBSSz}.
The key idea is to use the linear sets investigated in Section 6.

Let $U_1,\ldots,U_t$ be $\fq$-subspaces of $\fqn$ seen as an $n$-dimensional $\fq$-vector space such that
\[ U_1 \oplus \ldots \oplus U_t=\fqn. \]
Every vector $\mathbf{v} \in \fqn$ can be written uniquely as
\[ \mathbf{v}=\mathbf{v}_1+\ldots+\mathbf{v}_t, \]
where $\mathbf{v}_i \in U_i$ for every $i \in \{1,\ldots,t\}$. The vector $\mathbf{v}_i$ will be called the $i$-th \textbf{component} of $\mathbf{v}$ (with respect to $U_1,\ldots,U_t$) and the map $p_i\colon \fqn \rightarrow U_i$ such that $p_i(\mathbf{v})=\mathbf{v}_i$ is called the $i$-th \textbf{projection map}.

We also recall the definition of dual bases. Two ordered $\F_{q}$-bases $\mathcal{B}=(\xi_0,\ldots,\xi_{n-1})$ and $\mathcal{B}^*=(\xi_0^*,\ldots,\xi_{n-1}^*)$ of $\F_{q^n}$ are said to be \textbf{dual bases} if $\mathrm{Tr}_{q^n/q}(\xi_i \xi_j^*)= \delta_{ij}$, for $i,j\in\{0,\ldots,n-1\}$, where $\delta_{ij}$ is the Kronecker symbol. It is well known that for any $\F_q$-basis $\mathcal{B}=(\xi_0,\ldots,\xi_{n-1})$ there exists a unique dual basis $\mathcal{B}^*=(\xi_0^*,\ldots,\xi_{n-1}^*)$ of $\mathcal{B}$, see e.g.\ \cite[Definition 2.30]{lidl_finite_1997}.

\begin{lemma}\label{lem:projecmaps}
Let $U_1,\ldots,U_t$ be $\fq$-subspaces of $\fqn$ seen as an $n$-dimensional $\fq$-vector space such that
\[ U_1 \oplus \ldots \oplus U_t=\fqn. \]
Denote by $k_i$ the dimension of $U_i$ and let $\mathcal{B}_i=(\xi_{i,0},\ldots,\xi_{i,k_i-1})$ an ordered $\fq$-basis of $U_i$ for every $i \in \{1,\ldots,t\}$.
Let $\mathcal{B}^*=(\xi_{1,0}^*,\ldots,\xi_{1,k_1-1}^*,\ldots,\xi_{t,0}^*,\ldots,\xi_{t,k_t-1}^*)$ be the dual basis of $\mathcal{B}=(\xi_{1,0},\ldots,\xi_{1,k_1-1},\ldots,\xi_{t,0},\ldots,\xi_{t,k_t-1})$.
Then 
\[ p_i(x)=\sum_{j=0}^{k_i-1} \xi_{i,j}\mathrm{Tr}_{q^n/q}(\xi_{i,j}^*x)=\sum_{\ell=0}^{n-1}\left( \sum_{j=0}^{k_i-1} \xi_{i,j} \xi_{i,j}^{*q^\ell} \right)x^{q^\ell}. \]
\end{lemma}
\begin{proof}
Since $U_1 \oplus \ldots \oplus U_t=\fqn$, every $x \in \fqn$ can be uniquely written as
\[ x=x_1+\ldots+x_t, \]
with $x_i \in U_i$ for every $i \in \{1,\ldots,t\}$.
Let $i \in \{1,\ldots,t\}$ and let
\[ \overline{p}(x)=\sum_{j=0}^{k_i-1} \xi_{i,j}\mathrm{Tr}_{q^n/q}(\xi_{i,j}^*x). \]
Then it is easy to see that
\[ \overline{p}(\xi_{\ell,m})=\sum_{j=0}^{k_i-1} \xi_{i,j}\mathrm{Tr}_{q^n/q}(\xi_{i,j}^*\xi_{\ell,m})= 
\left\{
\begin{array}{ll}
\xi_{\ell,m} & \text{if}\,\,\ell=i,\\
0 & \text{otherwise},
\end{array}
\right.
\]
for any $(\ell,m)$ such that $\ell \in \{1,\ldots,t\}$ and $m \in \{0,\ldots,k_\ell-1\}$.
This means that $\overline{p}(x)$ coincides with the $i$-th projection map.
\end{proof}

For linear sets $L_U$ in $\PG(r-1,q^n)$ of rank $n$ containing $r$ independent points $P_1,\ldots,P_r$ such that
\[ w_{L_U}(P_1)+\ldots+w_{L_U}(P_r)=n, \]
we can determine $r$ linearized polynomials representing it.

\begin{theorem}
Let $L_W$ be an $\fq$-linear set in $\mathrm{PG}(r-1,q^n)$ of rank $n$ containing $r$ independent points $P_1=\langle \mathbf{v}_1\rangle_{\fqn},\ldots,P_r=\langle \mathbf{v}_r\rangle_{\fqn}$ such that
\[ w_{L_W}(P_1)+\ldots+w_{L_W}(P_r)=n. \] 
Then $L_W$ is $\mathrm{PGL}(r,q^n)$-equivalent to 
\[ L_p=\{ \langle (x, p_2(x),\ldots,p_r(x))\rangle_{\fqn} \colon x \in \fqn^* \}, \]
where $p_i(x)$ is the $i$-th projection map with respect to $r$ $\fq$-subspaces $U_1,\ldots,U_r$ of $\fqn$ such that
\[ U_1\oplus\ldots\oplus U_r=\fqn. \]
Moreover, if $k_i$ is the dimension of $U_i$ and $\mathcal{B}_i=(\xi_{i,0},\ldots,\xi_{i,k_i-1})$ is an ordered $\fq$-basis of $U_i$ for every $i \in \{1,\ldots,t\}$, then
\[ p_i(x)=\sum_{j=0}^{k_i-1} \xi_{i,j}\mathrm{Tr}_{q^n/q}(\xi_{i,j}^*x)=\sum_{\ell=0}^{n-1}\left( \sum_{j=0}^{k_i-1} \xi_{i,j} \xi_{i,j}^{*q^\ell} \right)x^{q^\ell}, \]
for every $i \in \{2,\ldots,r\}$, where
$\mathcal{B}^*=(\xi_{1,0}^*,\ldots,\xi_{1,k_1-1}^*,\ldots,\xi_{t,0}^*,\ldots,\xi_{t,k_t-1}^*)$ is the dual basis of $\mathcal{B}=(\xi_{1,0},\ldots,\xi_{1,k_1-1},\ldots,\xi_{t,0},\ldots,\xi_{t,k_t-1})$.
\end{theorem}
\begin{proof}
By Proposition \ref{prop:structure}, $L_W$ is $\mathrm{PGL}(r,q^n)$-equivalent to $L_U$ where 
$U=U_1\times \ldots \times U_r,$
for some $\fq$-subspaces $U_1,\ldots,U_r$ of $\fqn$ such that $U_i\cap \langle U_j \colon j \in\{1,\ldots,r\}\setminus\{i\} \rangle_{\fq}=\{\mathbf{0}\}$ for any $i\in \{1,\ldots,r\}$. 
Let $f$ be the $\fqn$-linear isomorphism of $\fqn^r$ determined by the matrix
\[ \left( \begin{array}{ccccc} 
1 & 1 & \cdots & 1 & 1 \\
0 & 1 & \cdots & 0 & 0 \\
\vdots & \vdots &  & \vdots & \vdots \\
0 & 0 & \cdots & 0 & 1 \\
\end{array}\right), \]
then
\[ f(U)=\{(x,p_2(x),\ldots,p_r(x)) \colon x \in \fqn \}, \]
and the assertion then follows by Lemma \ref{lem:projecmaps}.
\end{proof}

\end{document}